\documentclass[a4paper,10pt]{article} 
\usepackage[T1]{fontenc}
\usepackage{ucs}
\usepackage[utf8x]{inputenc}
\usepackage{amssymb,amsfonts,amsmath}
\usepackage{amsthm}   
\usepackage{verbatim}
\usepackage{epsfig}
\usepackage{mathrsfs}
\usepackage{enumerate}
\usepackage[ruled,vlined,linesnumbered]{algorithm2e}
\usepackage{color}
\usepackage{changes}
\usepackage{hyperref}

\usepackage{pgfplots}
\pgfplotsset{compat=newest}
\usetikzlibrary{plotmarks}
\usepackage{grffile}

\newcommand{\R}{\mathbb{R}}
\newcommand{\E}{\mathcal{E}}
\newcommand{\calH}{\mathcal{H}}
\newcommand{\calP}{\mathcal{P}}
\newcommand{\N}{\mathbb{N}}
\newcommand{\C}{\mathbb{C}}

\newcommand{\vect}{\mathrm{span}}

\newcommand{\sprod}[1]{\left\langle #1 \right\rangle}

\newtheorem{lemma}{Lemma}
\newtheorem{theorem}{Theorem}
\newtheorem{proposition}{Proposition}

\newtheorem{remark}{Remark}

\newtheorem{assumption}{Assumption}

\newcommand{\calM}{\mathcal{M}}
\newcommand{\M}{\mathcal{M}}
\newcommand{\B}{\mathcal{B}}
\newcommand{\Q}{\mathcal{Q}}
\newcommand{\Id}{\mathrm{Id}}

\newcommand{\calD}{\mathcal{D}}

\newcommand{\calC}{\mathcal{C}}

\renewcommand{\N}{\mathbb{N}}
\renewcommand{\ker}{\mathrm{ker} \,}
\newcommand{\ran}{\mathrm{ran} \,}

\newcommand{\norm}[1]{\left\Vert #1 \right\Vert}
\newcommand{\supp}{\mathrm{supp}}
\newcommand{\abs}[1]{\left\vert #1 \right\vert}
\newcommand{\st}{\text{ subject to }}
\newcommand{\sse}{\subseteq}
\newcommand{\set}[1]{\left\{ #1 \right \}}

\newcommand{\wstarto}{\stackrel{*}{\rightharpoonup}}
\newcommand{\wto}{{\rightharpoonup}}

\DeclareMathOperator{\esssup}{ess sup}

\title{Exact solutions of infinite dimensional total-variation regularized problems.}

\author{
Axel Flinth \& Pierre Weiss \\
}

\begin{document}

\maketitle

\begin{abstract}
We study the solutions of infinite dimensional linear inverse problems over Banach spaces.
The regularizer is defined  as the total variation of a linear mapping of the function to recover, while the data fitting term is a near arbitrary function.
The first contribution describes the solution's structure: we show that under mild assumptions, there always exists an $m$-sparse solution, where $m$ is the number of linear measurements of the signal. 
Our second contribution is about the computation of the solution. While most existing works first discretize the problem, we show that exact solutions of the infinite dimensional problem can be obtained by solving one or two consecutive finite dimensional convex programs depending on the measurement functions structures. These results extend recent advances in the understanding of total-variation regularized inverse problems.
\end{abstract}

\section{Introduction}\label{sec:intro}
Let $u\in \B$ be a signal in some vector space $\B$ and assume that it is probed indirectly, with $m$ corrupted linear measurements: 
\begin{align*}
 b=P(Au),
\end{align*}
where $A:\B\to \R^m$ is a measurement operator defined by $(Au)_i=\langle a_i , u\rangle$, each $a_i$ being an element in $\B^*$, the dual of $\B$. The mapping $P:\R^m\to\R^m$ denotes a perturbation of the measurements, such as quantization, additional Gaussian or Poisson noise, or any other common degradation operator.
Inverse problems consist in estimating $u$ from the measurements $b$.
Assuming that $\dim(\B)>m$, it is clearly impossible to recover $u$ knowing $b$ only. Hence, various regularization techniques have been proposed to stabilize the recovery. 

Probably the most well known and used example is Tikhonov regularization \cite{kaltenbacher2008iterative}, which consists in minimizing quadratic cost functions. The regularizers are particularily appreciated for their ease of analysis and implementation.
Over the last 20 years, sparsity promoting regularizers have proved increasingly useful, especially when the signals to recover have some underlying sparsity structure. 
Sparse regularization can be divided into two categories: the analysis formulation and the synthesis formulation. 

The analysis formulation consists in solving optimization problems of the form
\begin{align}\label{eq:mainproblem}
 \inf_{u\in \B} J(u):=f_b\left(Au \right) + \|Lu\|_{TV},
\end{align} 
where $f_b:\R^m \to \R\cup\{+\infty\}$ is an application dependent data fidelity term and $L:\B \to \E$ is a linear operator, mapping $\B$ to some space $\E$ such as $\R^n$, the space of sequences in $\ell_1$ or the space of Radon measures $\M$. The total variation norm $\|\cdot\|_{TV}$ coincides with the $\ell_1$-norm when $\E$ is discrete, but it is more general since it also applies to measures.

The synthesis formulation on its side consists in minimizing 
\begin{align}\label{eq:synthesis}
 \inf_{\mu \in \E} f_b\left(A D \mu\right) + \|\mu \|_{TV},
\end{align}
where $D:\E\to \B$ is the linear synthesis operator, also called dictionary. The estimate of $u$ in that case reads $\hat u = D\hat \mu$, where $\hat \mu$ is a solution of \eqref{eq:synthesis}.

Problems \eqref{eq:mainproblem} and \eqref{eq:synthesis} triggered a massive interest from both theoretical and practical perspectives. Among the most impressive theoretical results, one can cite the field of compressed sensing \cite{candes2006robust} or super-resolution \cite{candes2014towards,duval2015exact}, which certify that under suitable assumptions, the minimizers of \eqref{eq:mainproblem} or \eqref{eq:synthesis} coincide with the true signal $u$. 

Most of the studies in this field are confined to the case where both $\B$ and $\E$ are finite dimensional \cite{candes2006robust,donoho2006compressed,elad2007analysis,foucart2013mathematical}. 
In the last few years, some efforts have been provided to get a better understanding of  \eqref{eq:mainproblem} and \eqref{eq:synthesis} where $\B$ and $\E$ are sequence spaces \cite{adcock2016generalized,adcock2017breaking,unser2016representer,traonmilin2017compressed}.
Finally, a different route, which will be followed in this paper, is the case where $\E=\M$, the space of Radon measures on a continuous domain. In that case, problems \eqref{eq:mainproblem} and \eqref{eq:synthesis} are infinite dimensional problems over measure spaces. One instance in that class is that of total variation minimization (in the PDE sense \cite{ambrosio2000functions}, that is the total variation of the distributional derivative), which became extremely popular in the field of imaging since its  introduction in \cite{rudin1992nonlinear}. 
There has been surge of interest in understanding the fine properties of the solutions in this setting, with many significant results \cite{bredies2013inverse,candes2014towards,tang2013compressed,duval2015exact,chambolle2016geometric,unser2016splines}.
The aim of this paper is to continue these efforts by bringing new insights in a general setting.

\paragraph{Contributions and related works}

The main contributions are twofold: one is about the structure of the solutions of \eqref{eq:mainproblem}, while the other is about how to numerically solve this problem without discretization. The results directly apply to problem \eqref{eq:synthesis} since, with regards to our concerns, the synthesis problem \eqref{eq:synthesis} is a special case of the analysis problem \eqref{eq:mainproblem}. It indeed suffices to take $L=\Id$ and $\B=\M$ for \eqref{eq:synthesis} to be an instance of \eqref{eq:mainproblem}. Notice however that in general, the two approaches should be studied separately \cite{elad2007analysis}.  

On the theoretical side, we provide a theorem characterizing the structure of the solutions of problem \eqref{eq:mainproblem} under certain assumptions on the operator $L$. Roughly speaking, this theorem states that there always exist $m$-sparse solutions. The precise meaning of this claim will be clarified in Theorem \ref{thm:structure}. This result is strongly related and was actually motivated by \cite{unser2016splines}. In there, the authors restrict their study to certain stationary operators $L$ over spaces of functions defined on $\Omega=\R^d$. Their main result states that in that case, generalized splines with $m$ knots actually describe the whole set of solutions. 
Similar results \cite{fisher1975spline} were actually obtained much earlier on bounded domains and seem to have remained widely ignored until they were revitalised by Unser-Fageot-Ward.
The value of our result lies in the fact that it holds for more general classes of operators $L$, spaces $\B$ , domains $\Omega$ and functions $f_b$. Furthermore, the proof technique is different from \cite{unser2016splines}: it is constructive and presumably applicable to wider settings.

On the numerical side, let us first emphasize that in an overwhelming number of works, problem \eqref{eq:mainproblem} is solved by first discretizing the problem to make it finite dimensional and then approximate solutions are found with standard procedures from convex programming. Theories such as $\Gamma$-convergence \cite{braides2002gamma} then sometimes allow showing that as the discretization parameter goes to $0$, solutions of the discretized problem converge (in a weak sense) to the solutions of the continuous problem. In this paper, we show that under some assumptions on the measurement functions $(a_i)$, the infinite dimensional problem \eqref{eq:mainproblem} can be attacked directly without discretization: the resolution of one or two consecutive finite dimensional convex programs allows recovering \emph{exact solutions} to problem \eqref{eq:mainproblem} or \eqref{eq:synthesis}. The structure of the convex programs depend on the structure of measurement vectors. Once again, this result is strongly related to recent advances. For instance, it is shown in \cite{candes2014towards,tang2013compressed} that a specific instance of \eqref{eq:synthesis} with $L=\Id$ can be solved exactly thanks to semi-definite relaxation or Prony type methods when the signal domain is the torus $\Omega=\mathbb{T}$ and the functions $(a_i)$ are trigonometric polynomials. 
Similar results were obtained in \cite{de2017exact} for more general semi-algebraic domains using Lasserre hierarchies \cite{lasserre2001global}. Once again, the value of our paper lies in the fact that it holds for near arbitrary convex functions $f_b$ and for a large class of operators $L$ such as the derivative. To the best of our knowledge, the only case considered until now was $L=\Id$.
In addition, our results provide some insight on the standard minimization strategy: we show that it corresponds to solving a different infinite dimensional problem \emph{exactly}, where the sampling functions are piecewise linear. We also show that the solution of the standard discretization can be made sparser by merging Dirac masses located on neighboring grid points.

\section{Main results}\label{sec:main}

\subsection{Notation}

In all of the paper, $\Omega\subseteq \R^d$ denotes an open subset either bounded or unbounded. 
The space of distributions on $\Omega$ is denoted $\calD^*(\Omega)$.
We let $\M(\Omega)$ denote the set of Radon measures on $\Omega$, i.e. the dual $\calC_0(\Omega)^*$ of $\calC_0(\Omega)$, the space of continuous functions on $\Omega$ vanishing at infinity:
\begin{multline*}
 \calC_0(\Omega)=\begin{cases} f:\Omega\to \R, &f \textrm{ continuous}, \\ &\forall \epsilon>0, \exists C\subset \Omega \textrm{ compact }, \forall x \notin C, \abs{f(x)}<\epsilon \end{cases}\Bigg\}.
\end{multline*}
We will throughout the whole paper view $(\M(\Omega),\|\cdot\|_{TV})$ as a Banach space, and not, as often is done, as a locally convex space equipped with the weak-$*$-topology. When we do this, $\calC_0(\Omega)$ is a subset, and not the whole of, the dual $\M^*$ of $\M$ (as it would have been if we have viewed $\M$ as a locally convex space).


Let $J:\B\to \R\cup \{+\infty\}$ denote a convex lower-semicontinuous function. We let $J^*$ denote its Fenchel conjugate and $\partial J(u)$ denote its subdifferential at $u\in \B$. 
Let $X\subset \mathcal{E}$ be a subset of some vector space $\mathcal{E}$. The indicator function of $X$ is defined for all $e\in\mathcal{E}$ by:
\begin{equation*}
\iota_X(e) = \begin{cases}
              0 & \textrm{if } e\in X \\
              +\infty & \textrm{otherwise}.
             \end{cases} 
\end{equation*}
We refer the reader to \cite{ekeland1999convex} for more insight on convex analysis in vector spaces.

\begin{remark}\label{rem:Omega}
All the results in our paper hold when $\Omega$ is a separable, locally compact topological space such as the torus $\mathbb{T}=\R\backslash \N$. The proofs require minor technical amendments related to the way the space is discretized. We chose to keep a simpler presentation in this paper.
\end{remark}

\subsection{Assumptions} \label{sec:assump}

Let us describe the setting in which we will prove the main result in some detail. Let $L: \calD'(\Omega) \to \calM(\Omega)$ be a continuous linear operator defined on the space of distributions $\calD'(\Omega)$.
Consider the following linear subspace of $\calD^*(\Omega)$
\begin{align*}
	\B^{\circ}= \set{u \in \calD'(\Omega) \, \vert \, Lu \in \calM(\Omega)}.
\end{align*}
Now, let $\norm{\cdot}_K$ be a semi-norm on $\B^\circ$, which restricted to $\ker L$ is a norm. We define 
\begin{align*}
	\B = \set{u \in \calD'(\Omega) \, \vert \, Lu \in \M(\Omega) , \norm{u}_K < \infty},
\end{align*}
and equip it with the norm $\norm{u}= \norm{Lu}_{TV} + \norm{u}_K$. We will assume that
\begin{assumption}[Assumption on $\B$] \label{assump:B}
$\B$ is a Banach space.
\end{assumption}
 We will make the following additional structurial assumptions on the map $L$:

\begin{assumption}[Assumptions on $L$]\label{assump:L}\ 
\begin{itemize} 
 \item The kernel of $L$  has a complementary subspace, i.e. a closed subspace  $V$ such that $\ker  L \oplus V = \B$.
 \item The range of $L$ is closed, and has a complementary subspace $W$, i.e., $\ran L \oplus W = \M$.
 \end{itemize}
\end{assumption}

An important special case of operators satisfying the assumption \ref{assump:L} are \emph{Fredholm operators} for which the space $W$ complementary to $\ran L$ is finite-dimensional, and $\ker L$ is itself finite dimensional, see e.g. \cite[Lemma 4.21]{rudin1991functional}.

The restriction $L|_V$ of $L : V \to \ran L$ is a bijective operator, and therefore has a continuous inverse $(L\vert_V)^{-1}$, by the continuous inverse theorem. With the help of this inverse, we can define a \emph{pseudoinverse} $L^+: \calM \to \B$ through
\begin{align*}
	L^+ = j_V (L\vert_V)^{-1} \Pi_{\ran L},
\end{align*}
where $j_V$ denotes the injection $V \hookrightarrow \B$ and $\Pi_{\ran L}$ the projection from $M$ to $\ran L$. Both of these as well as $(L\vert_V)^{-1}$ are continuous, so that $L^+$ is continuous.

We will furthermore have to restrict the functionals $a_i$ used to probe the signals slightly.
\begin{assumption}[Assumption on $a_i$] \label{assump:a} \

The functionals $a_i \in \B^*(\Omega)$ have the property that $(L^+)^*a_i \in \calC_0(\Omega)$. That is, there exist functions $\rho_i \in\calC_0(\Omega)$ with
	\begin{align*}
		\forall \mu \in \M: \, \sprod{(L^+)^*a_i, \mu}=\int_\Omega \rho_i(x) d\mu(x).
	\end{align*}
\end{assumption}
This assumption may seem a bit artificial, but we will see that it is crucial, both in the more theoretical first part of the paper, as well as in the second one dealing with the numerical resolution of the problem. Furthermore, it is equivalent to an assumption in the main result of \cite{unser2016splines}, as will be made explicit in the sequel. 

Until now, we have not touched upon the properties of the function $f_b$. We do this implicitly with the following condition:

\begin{assumption}[Solvability Assumption] \label{assump:solvability} The problem \eqref{eq:mainproblem} has at least one solution.
\end{assumption}

This assumption is of course necessary in order to make questions about the structure of the solutions of \eqref{eq:mainproblem} to make sense at all. A myriad of problems have this property, as the following simple proposition shows:

\begin{proposition} \label{prop:solvability}
	Assume that $f_b$ is lower semi-continuous and coercive (i.e. $\lim_{\norm{x}_2\to \infty} f(x) = \infty$), and that $f_b \circ A$ has a non-empty domain. Then, under assumptions \ref{assump:B}, \ref{assump:L} and \ref{assump:a}, the problem \eqref{eq:mainproblem} has a solution.
\end{proposition}

The proof, which relies on standard arguments, can be found in Section \ref{sec:misc}. Let us here instead argue that the assumptions in  \ref{prop:solvability} are quite light and cover many common data fidelity terms as exemplified below.

\begin{description}
 \item[Equality constraints] This case corresponds to 
 \begin{equation}\label{eq:equalityconstraint}
  f_b(x)=\iota_{\{b\}}(x) = 
  \begin{cases}0 & \textrm{ if } x=b \\
  +\infty & \textrm{ otherwise.}
  \end{cases}  
 \end{equation}
This data fidelity term is commonly used when the data is not corrupted. A solution exists if $b\in \ran(A)$.
The two super-resolution papers \cite{candes2014towards,tang2013compressed} use this assumption.
 \item[Quadratic]  The case $f_b(x) = \frac{\lambda}{2}\|C^{-1}(x-b)\|_2^2$, where $\lambda>0$ is a data fit parameter, is commonly used when the data suffers from additive Gaussian noise with a covariance matrix $C$.
 \item[$\ell_1$-norm] When data suffers from outliers, it is common to set $f_b(x) = \lambda\|x-b\|_1$, with $\lambda>0$.
 \item[Box constraints] When the data is quantized, a natural data fidelity term is a box constraint of the following type 
 \begin{equation*}
f_b(x)=\begin{cases}
0 & \textrm{ if } \|C(x-b)\|_\infty \leq 1 \\ 
+\infty & \textrm{ otherwise,}                                                                   
\end{cases}
\end{equation*}
 where $C\in \R^{m\times m}$ is a diagonal matrix with positive entries.

\item[Phase Retrieval] Many non-convex functions $f_b$ fulfill our assumptions. In particular, any of the above fidelity terms can be combined with the (pointwise) absolute value to yield a feasible function $f_b$, i.e. for instance
 \begin{align*}
  f_b(x)=\iota_{\{b\}}(\abs{x}) = 
  \begin{cases}0 & \textrm{ if } \abs{x}=b \\
  +\infty & \textrm{ otherwise.}
  \end{cases}  .
 \end{align*}
 Such functions appear in the \emph{phase retrieval problem}, where one tries to reconstruct a signal $u$ from absolute values of type $\abs{Au}$.
 
\end{description}

\subsection{Structure of the solutions}

We are now ready to state the first important result of this paper.

\begin{theorem}\label{thm:structure}
	Under assumptions \ref{assump:B}, \ref{assump:L}, \ref{assump:a} and \ref{assump:solvability}, problem \eqref{eq:mainproblem} has a solution of the form
	\begin{align*}
		\hat{u} = u_K+ \sum_{k=1}^p d_k L^{+} \delta_{x_k},
	\end{align*}
	with $p\leq \overline{m} = m - \dim (A^* (\ker L))$, $u_K \in \ker L$, $d=(d_k)_{1\leq k \leq p}$ in  $\R^p$ and $X=(x_k)_{1\leq k \leq p}$ in $\Omega^p$.
\end{theorem}

The proof of this theorem consists of three main steps. We provide the first two below, since they are elementary and provide some insight on the theorem. The last step appears in many works. We provide an original proof in the appendix.

\begin{proof}
\ 

\noindent\textbf{Step 1:} In this step, we transform the data fitting $f_b$ into an equality constraint.
 To see why this is possible, let $\overline{u}$ be a solution of the problem \eqref{eq:mainproblem}. Then any solution of the problem
\begin{align*}
	\min_{u \in \B} \norm{Lu}_{TV}  \st Au = A\overline{u}=:y
\end{align*}
will also be a solution $\hat u$ of \eqref{eq:mainproblem}, since it satisfies $f_b(A\hat u)=f_b(A\overline{u})$ and $\norm{L\overline{u}}_{TV}=\norm{L\hat u}_{TV}$. Those two equalities are required, otherwise, $\overline{u}$ would not be a solution since $J(\hat u)<J(\overline{u})$.

\paragraph{Step 2:} In this step, we show that it is possible to discard the operator $L$. 
To see this, notice that since every $u \in \B$ can be written as $L^+\mu + u_K$ with $\mu \in \calM$ and $u_K \in \ker L$. Therefore, we have
 \begin{align*}
 &\left(\min_{u\in \B} \|Lu\|_{TV} \st Au=y \right) \\
  & \quad= \left(\min_{\substack{u_K\in \ker(L), \mu\in \M  }} \|\mu\|_{TV} \st A(u_K+L^+\mu)=y \right)
 \end{align*}
 
 Now, set $X=A\ker(L)$. Since $X$ is a finite-dimensional subspace of $\R^m$, we may decompose $y=y_X+y_{X^\perp}$, with $y_X \in X$ and $y_{X^\perp} \in X^\perp$, the orthogonal complement of $X$ in $\R^m$. Notice that for every $\mu \in \M$, there exists a $u_K \in \ker L$ with $A(u_K+L^+\mu)=y$ if and only if $\Pi_{X^\perp} AL^+\mu = y_{X^\perp}$. Hence, the above problems can be simplified as follows
 \begin{equation} \label{eq:reducedProblem}
	 \min \norm{\mu}_{TV} \st H\mu = y_{X^\perp},
 \end{equation}
with  $H: \M \to X^\perp$, $H= \Pi_{X^\perp} A^*L^+$, with $\dim X^\perp= \overline{m}$.

\paragraph{Step 3:} The last step consists in proving that the problem \eqref{eq:reducedProblem} has a solution of the form $\sum_{k=1}^p d_k \delta_{x_k}$. 
This result is well-known when $\Omega$ is a countable set, see e.g. \cite{unser2016representer}.
It is also available in infinite dimensions on compact domains. We refer to \cite{fisher1975spline} for instance, for an early proof, based on the Krein-Milmann theorem. 
We propose an alternative strategy in the appendix based on a discretization procedure.
\end{proof}

\begin{remark}
In \cite{fisher1975spline,unser2016splines}, the authors further show that the extremal points of the solution set are of the form given in Theorem \ref{thm:structure}, if $f_b$ is the indicator function of a closed convex set. Their argument is based on a proof by contradiction. Following this approach, it is possible to prove the same result in our setting. We choose not to carry out the details about this since we also wish to cover nonconvex problems.
\end{remark}

Before going further, let us show some consequences of this theorem.

\subsubsection{Example 1: $L=\Id$ and the space $\M$}\label{subsec:example1}

Probably the easiest case consists in choosing an arbitrary open domain $\Omega\subseteq \R^d$, to set $\B = \M(\Omega)$ and $L=\Id$. In this case, all the assumptions \ref{assump:L} on $L$ are trivially met. We have $\ran \, \Id= \M(\Omega)$, $\ker \, \Id = \set{0}$ and $\Id^+=\Id$. Therefore, Theorem \ref{thm:structure} in this specific case guarantees the existence of a minimizer of \eqref{eq:mainproblem} of the form
\begin{align*}
		\hat{\mu} = \sum_{k=1}^p d_k \delta_{x_k},
\end{align*}
with $p\leq m$.
The assumption \ref{assump:a} in this case simply means that the functionals $a_i$ can be identified with continuous operators vanishing at infinity.

Note that the synthesis formulation \eqref{eq:synthesis} can be seen as a subcase of this setting. The structure of the minimizing measure in Theorem \ref{thm:structure} implies that the signal estimate $\hat u$ has the following form
\begin{align*}
	\hat{u} = D\hat{\mu} = \sum_{k=1}^p d_k D\delta_{x_k}.
\end{align*}
The vectors $(D\delta_{x})_{x \in \Omega}$ can naturally be interpreted as the atoms of a dictionary. Hence, Theorem \ref{thm:structure} states that there will always exist at least one estimate from the synthesis formulation which is sparsely representable in the dictionary $(D\delta_{x})_{x \in \Omega}$.

\subsubsection{Example 2: Spline-admissible operators and their native spaces}
The authors of \cite{unser2016splines} consider a generic operator $L$ defined on the space of tempered distributions $\mathcal{S}'(\R^d)$  and mapping into $\M(\R^d)$, which is
\begin{itemize}
\item Shift-invariant, 
\item  for which there exists a function $\rho_L$ (a generalized spline) of polynomial growth, say
\begin{align} \label{eq:growth}
	\esssup_{x \in \R^d} \abs{\rho_L(x)} (1+\norm{x})^{-r}<+\infty,
\end{align} obeying $L\rho_L = \delta_0$.
\item The space of functions in the kernel of $L$ obeying the growth estimate \eqref{eq:growth} is finite dimensional.
\end{itemize} 
The authors call such operators \emph{spline-admissible}. A typical example is the distributional derivative $D$ on $\Omega=\R$.
For each such operator $L$, they define a space $\M_L(\R^d)$ as the set of functions $f$ obeying the growth estimate \eqref{eq:growth} while still having the property $Lf \in \M(\R^d)$. The norm on $\M_L$ is as in our formulation, whereby $\norm{\cdot}_K$ is defined through a dual basis of a (finite) basis of $\ker L$.

They go on to prove that $\M_L(\R^d)$ is a Banach space, which has a separable predual $\calC_L(\R^d)$, and (in our notation) assume that the functionals $a_i \in M_L^*(\R^d)$ can be identified with elements of $\calC_L(\R^d)$.

It turns out that using this construction, the operator $L$ and functionals $(a_i)$ obey the assumptions \ref{assump:L} and \ref{assump:a}, respectively.

\begin{proposition} \label{prop:Unser} \ 
\begin{itemize}
	\item The operator $L : \M_L(\R^d) \to \M(\R^d)$ is Fredholm. In fact, $\ran L$ is even equal to $\M(\R^d)$.
	\item The functionals $a_i\in M_L^*(\R^d)$ obey assumption \ref{assump:a}. In fact, we even have
	\begin{align*}
		(L^+)^*a \in \calC_0(\R^d) \ \Longleftrightarrow \ a \in \calC_L(\R^d).
	\end{align*}
	\end{itemize}
\end{proposition}
Hence, the assumptions in \cite{unser2016splines} are a special case of the ones used in this paper.

\subsubsection{Example 3: More general differential operators and associated spaces} \label{ex:DiffOp}
The inclusion of operators with infinite dimensional kernel allows us to treat differential operators in a bit more streamlined way than above, in particular removing the restricted growth conditions. Let $\Omega$ be an open subset of $\R^d$ and $P(D)$ a differential operator on $\Omega$, i.e. an expression of the form
\begin{align*}
	P(D) = \sum_{\abs{\alpha} \leq K} p_\alpha D^\alpha,
\end{align*}
where $D^\alpha = \partial_{\alpha_1} \cdot \partial_{\alpha_j}$ is a partial derivative operator and $p_\alpha$ are measurable functions on $\Omega$. Note that $P(D)$ does not need to be shift invariant (if $\Omega \neq \R^d$, shift-invariance is not even possible to define).

In order to define the norm of functions in the kernel  of $L=P(D)$ properly, which we will not assume to satisfy any growth conditions, we assume that there exists a bounded subset $K \sse \Omega$ with the following \emph{continuation property:}
\begin{assumption}[Continuation property] \ \label{assumption:continuation}
For each distribution $u \in \calD'(K)$ with $P(D)u=0$, there exists exactly one $\widehat{u} \in \calD'(\Omega)$ with $P(D)\widehat{u} = 0$ in $\Omega$ and $\widehat{u}=u$ in $K$. 
 \end{assumption}
 We will see that for a large class of elliptic operators, we can choose $K$ to be any bounded set with non-empty interior and smooth boundary. These conditions will furthermore in particular prove that $\norm{\cdot }_{\calM(K)}$ is a seminorm on a space $\B$, which restricted to $\ker P(D)$ is a norm.

The fundamental assumption we will make is the following:
\begin{assumption}[Green function hypothesis] \label{assumption:PD} \
For each $x \in \Omega$, there exists a solution $u_x \in \calC(\Omega)$ of the problem
\begin{align}	\label{eq:fundamentalSolution}
P(D)u_x = \delta_x.
\end{align}
We also assume that the map $\Omega \ni x \mapsto u_x \in \calC(\Omega)$ is continuous and bounded, i.e. $\sup_{x \in \Omega} \norm{u_x}_\infty < \infty$.
\end{assumption}

Now we define, inspired by the native spaces $\calM_L$ from above, a space $\B_P$, which $P(D)$ naturally sends to $\calM(\Omega)$:
\begin{align*}
	\B_P = \set{u \in \calD'(\Omega) \, \vert \, P(D) u \in \calM(\Omega), u\vert_K \in \M(K)}.
\end{align*}

\begin{lemma} \label{lem:BP}
	Under assumptions \ref{assumption:continuation} and \ref{assumption:PD}, the following holds: The expression
	\begin{align*}
		\norm{u}_{\B_P} = \norm{ P(D) u }_{TV}+ \norm{u \vert_K}_{TV}
	\end{align*}
	defines a norm on $\B_P$. $\B_P$ equipped with this norm is a Banach space, i.e., satisfies assumption \ref{assump:B}.
\end{lemma}

 We now prove that \ref{assumption:PD} implies that $P(D)$ obeys the assumption \ref{assump:L}, and state a more specific one which implies that relatively general functionals $a$ obey assumption \ref{assump:a}. To simplify the formulation of it slightly, let us introduce the following notion: we say that a mapping $T : \Omega \to \calC(\Omega)$ vanishes at infinity on compact sets if for each compact subset $C \sse \Omega$, the function $x \mapsto \sup_{y \in C} \abs{T(x)(y)}$ vanishes at infinity.

\begin{proposition} \label{prop:DiffOp}
Under assumption \ref{assumption:PD}, $L=P(D)$ satisfies assumptions \ref{assump:L}. In particular, $\ran L= \M$, and the operator $L^+$ is given by
\begin{align} \label{eq:LPlusDiffOp}
	(L^+ \mu)(x) = \int_\Omega u_y (x) d\mu(y)
\end{align}
Furthermore, if $a$ is a functional of the type
\begin{align} \label{eq:afunctional}
	\sprod{a, u} = \int_{\Omega} a(x) u(x) dx,
\end{align}
with $a \in L^1(\Omega)$, we have
\begin{align*}
	((L^+)^*a) (x) = \int_{\Omega} u_{x}(y) a_i(y) dy
\end{align*}
$(L^+)^*a$ obeys the assumption \ref{assump:a} provided the map $ x \mapsto u_x$ vanishes at infinity on compact sets.
\end{proposition}

\begin{remark}
	Since we have assumed no growth restriction on the elements of $\ker P(D)$, in general, not every function $a \in L^1(\Omega)$ will cause \eqref{eq:afunctional} to define a functional on $\B_P$. However, \emph{if} this is the case for a \emph{specific} $a$, $(L^+)^*a$ will be well-defined and have the claimed form.
	
	An example of an additional assumption which will make \eqref{eq:afunctional} actually define a functional on $\B_P$ is that $a$ is continuous and has compact support inside $K$, since then
	\begin{align*}
		\abs{\sprod{a, u}} \leq \norm{a}_\infty \norm{u \vert_K}_{TV} \leq \norm{a}_\infty \norm{u}_{\B_P}.
	\end{align*}
\end{remark}

Let us now give a relatively general example of operators which satisfy the properties presented in Proposition \ref{prop:DiffOp}. It for instance includes poly-Laplacian operators $\Delta^k$ of sufficiently high order on for $\Omega \sse \R^d$, either bounded or equal to the entire space $\R^n$.

Let $k \in \N$ and $P(D)$ be a differential operator on $\Omega$   of the form
\begin{align} \label{eq:Elliptic}
	P(D) = \sum_{\abs{\alpha} = k} \sum_{\abs{\beta}=k} D^\beta(p_{\alpha,\beta}(x) D^\alpha)
\end{align}
for some bounded functions $p_{\alpha, \beta}\in \calC^k(\Omega)$. Also assume that $P(D)$ obeys the following ellipticity condition
\begin{align}
	\inf_{x \in \Omega} \inf_{\norm{\xi}_2=1} \sum_{\abs{\alpha}=k} \sum_{\abs{\beta}=k} p_{\alpha,\beta}(x) \xi^{\alpha+\beta}=: C >0. \label{eq:ellipticity}
\end{align}
\begin{proposition} \label{prop:Elliptic}
 Suppose that $k > \frac{d}{2}$. For either $\Omega$ bounded with Lipschitz domain or $\Omega=\R^d$, the following is true. Under the ellipticity assumption \eqref{eq:ellipticity}, the problem \eqref{eq:fundamentalSolution} admits for each $x \in \Omega$ a solution $u_x \in \calC(\Omega)$. The map $x \to u_x$ is furthermore vanishing at infinity on compact sets.
 
 Also, any set $K$  with non-empty interior and smooth domain obeys assumption \ref{assumption:continuation}.
\end{proposition}

\begin{remark}
	The assumption $k>d/2$ is crucial, since only then, we can guarantee that the solutions $u_x$ of \eqref{eq:fundamentalSolution} are continuous. Consider for instance the Laplacian operator $\Delta$ on $\R^d$ for $d \geq 2$. Then $k=1\leq d/2$ and 
	\begin{align*}
		u_x(y) = \begin{cases} \log( \norm{y-x}_2 ) & \text{ if } d =2, \\
		\norm{x-y}_2^{2-d}  & \text{ otherwise,} \end{cases}
	\end{align*}
	which are not continuous.
\end{remark}

\subsubsection{Example 4: $L=D$ and the space $BV(\rbrack 0,1 \lbrack )$}\label{subsec:example2}
Another important operator which is \emph{not} covered by Proposition \ref{prop:DiffOp} is the univariate derivative $D$ in the univariate case. In this case, the function $u_x$ in \eqref{eq:fundamentalSolution} is  equal to a shifted Heaviside function, which of course is not continuous. At least this operator can however still be naturally included in our framework, as we will show here.

We set $\Omega=]0,1[$.
The space $BV(\Omega)$ of bounded variation functions is defined by (see \cite{ambrosio2000functions}):
\begin{equation}
 BV(\Omega) = \{ u \in L^1(\Omega), Du \textrm{ is a Radon measure}, \|Du\|_{TV} <+\infty\},
\end{equation}
where $D$ is the distributional derivative. 
Using our notations, it amounts to taking $L=D$ and $\B=BV(\Omega)$.
For this space, we have $\ker L=\vect(1)$, the vector space of constant functions on $\Omega$. (Note that in fact, the norm $\norm{u}_{BV} = \norm{u}_1 + \norm{Du}_{TV}$ is of the general form described in the introduction).

\begin{lemma}\label{lem:Dplus}
For $L=D$ we have $\ran L=\M$, and for all $\mu \in \M$ and all $s\in [0,1]$,
\begin{equation}
 (L^+ \mu)(s) = \mu([0,s]) - \int_0^1 \mu([0,t])\,dt.
\end{equation}
In addition, for a functional $\xi \in BV(]0,1[)^*$ of the form
\begin{align*}
	\sprod{\xi, u} = \int_0^1 \xi(t) u(t) \,dt,
\end{align*}
with $\xi \in L^1(\Omega)$, we have $(L^+)^*\xi \in \calC_0(\Omega)$ and letting $\bar \xi=\int_{0}^1\xi(t)\,dt$, we have
\begin{equation}\label{eq:defDplusstar}
	((L^+)^*\xi)(s) = \int_0^s (\bar \xi - \xi(t)) \,dt.
\end{equation}
\end{lemma}

As can be seen, $L^+$ is simply a primitive operator. 
The elementary functions $L^+\delta_x$ are Heavyside functions translated at a distance $x$ from the origin.
Hence, Theorem \ref{thm:structure} states that there always exist total variation minimizers in 1D that can be written as staircase functions with at most $\overline{m}$ jumps. 
Note that in this case, the Heavyside functions coincide with the general splines introduced in \cite{unser2016splines}.

\subsubsection{An uncovered case: $L=\nabla$  and the space $BV(\rbrack 0,1 \lbrack^2 )$}\label{subsec:example3}

It is very tempting to use Theorem \ref{thm:structure} on the space $\B=BV(\rbrack 0,1 \lbrack^2 )$. As mentioned in the introduction, this space is crucial in image processing since its introduction in \cite{rudin1992nonlinear}. Unfortunately, this case is not covered by Theorem \ref{thm:structure}, since $L\B$ is then a space of vector valued Radon measures, and our assumptions only cover the case of scalar measures.

\subsection{Numerical resolution} \label{sec:Num}

In this section, we show how the infinite dimensional problem \eqref{eq:mainproblem} can be solved using standard optimization approaches. We will make the following additional assumption:
\begin{assumption}[Additional assumption on $f_b$] \
$f_b$ is convex  and lower semicontinuous.
\end{assumption}
Depending on the structure of the measurement functions $(a_i)$, we will propose to solve the primal problem \eqref{eq:mainproblem} directly, or to solve two consecutive convex problems: the dual and the primal.
We first recollect a few properties of the dual to shed some light on the solutions properties.

\subsubsection{The dual problem and its relationship to the primal}\label{sec:dual}

A natural way to turn \eqref{eq:mainproblem} into a finite dimensional problem is to use duality as shown in the following proposition.

\begin{proposition}[Dual of problem \eqref{eq:mainproblem}]\label{prop:dual}
Define $h:\M(\Omega) \to \R\cup \{+\infty\}$ by $h(\mu)=\|\mu\|_{TV} + \iota_{\ran L}(\mu)$.
Then, the following duality relationship holds:
 \begin{equation}\label{eq:dualbad}
  \min_{u \in \B} J(u) = \sup_{q \in \R^m, A^*q\in \ran L^*}  - h^*((L^+)^* A^*q) - f_b^*(q).
 \end{equation}
In the special case $\ran L = \M$, this yields
\begin{equation}\label{eq:dualgood}
	\min_{u \in \B} J(u) = \sup_{q \in \R^m, A^*q\in \ran L^*, \norm{(L^+)^*A^*q}_\infty \leq 1}  - f_b^*(q).
\end{equation}
 
In addition, let $(\hat u, \hat q)$ denote any primal-dual pair of problem \eqref{eq:dualbad}. The following duality relationships hold:
\begin{equation}
 A^* \hat q \in L^* \partial({\|\cdot\|_{TV}})(L\hat u) \label{eq:primaldual} \mbox{ and } -\hat q \in \partial f_b(A \hat u). 
\end{equation}
\end{proposition}

For a general operator $L$, computing $h^*$ may be out of reach, since the conjugate of a sum cannot be easily deduced from the conjugates of each function in the sum. Hence, we now focus on problem \eqref{eq:dualgood} corresponding to the case $\ran L=\M$. This covers at least the two important cases $L=\Id$ and $L=D$, as shown in examples \ref{subsec:example1} and \ref{subsec:example2}.

\begin{remark}
In general, the dual problem does not need to have a solution. A straightforward application of \cite[Th. 4.2]{BorweinLewis1992} however shows that if either of the two following conditions hold
\begin{enumerate}
	\item $\ran A$ intersects the relative interior of $ \mathrm{dom}f_b = \{
	q\in \R^m : f_b(q)<\infty\}$,
	\item $f_b$ is polyhedral (i.e. has a convex polyhedral epigraph) and $\ran A$ intersects $\mathrm{dom} f_b$,
\end{enumerate} 
the dual problem does have a solution. These conditions are mild: For all of the convex examples discussed in Section \ref{sec:assump}, the existence of a $u \in \B$ with $Au=b$ is sufficient for at least one of them to hold.
\end{remark}

Solving the dual problem \eqref{eq:dualgood} does not directly provide a solution for the primal problem \eqref{eq:mainproblem}.
The following proposition shows that it however yields information about the support of $L\hat u$, which is the critical information to retrieve. 

\begin{proposition}\label{prop:dualtoprimal}
Assume that $\ran L=\M$ and let $(\hat u,\hat q)$ denote a primal-dual pair of problem \eqref{eq:dualgood}.
Let $I(\hat q)=\{x\in \Omega, |(L^+)^*(A^*\hat q)|(x)=1\}$. Then 
\begin{equation}
\supp(L\hat u)\subseteq I(\hat q). 
\end{equation}
In particular, if $I(\hat q) = \{x_1,\hdots, x_p\}$, then $\hat u$ can be written as:
 \begin{equation}\label{eq:solutionfinite}
  \hat u = u_K + \sum_{k=1}^p d_k L^+\delta_{x_k}
 \end{equation}
 with $u_K\in \ker L$ and $(d_k)\in \R^p$. If problem \eqref{eq:mainproblem} admits a unique solution, then $p\leq \overline{m}$ and $\hat u$ is the solution in Theorem \ref{thm:structure}.
\end{proposition}

In the case where $I(\hat q)$ is a finite set, Proposition \ref{prop:dualtoprimal} can be used to recover a solution $\hat u$ from $\hat q$, by injecting the specific structure \eqref{eq:solutionfinite} into \eqref{eq:mainproblem}. Let $(\lambda_i)_{1\leq i\leq r}$ denote a basis of $\ker L$ and define the matrix
\begin{equation}\label{eq:matrixM}
 M=\begin{bmatrix}
    (\langle a_i, \lambda_k\rangle)_{1\leq i \leq m,1\leq k \leq r} , (\langle (L^+)^* a_i, \delta_{x_j}\rangle)_{1\leq i \leq m,1\leq j \leq p}
   \end{bmatrix}
\end{equation}
Then problem \eqref{eq:mainproblem} becomes a finite dimensional convex program which can be solved with off-the-shelf algorithms:
\begin{equation}\label{eq:primalmadefinite}
 \min_{c\in \R^r, d\in \R^p} f_b\left( M \begin{bmatrix}
                                          c \\ d
                                         \end{bmatrix}\right) + \|d\|_1.
\end{equation}

Overall, this section suggests the following strategy to recover $\hat u$: 
\begin{enumerate}
 \item Find a solution $\hat q$ of the finite dimensional dual problem \eqref{eq:dualgood}.
 \item Identify the support $I(\hat q)=\{x\in \Omega, |(L^+)^*(A^*\hat q)|(x)=1\}$. 
 \item If $I(\hat q)$ is finitely supported, solve the finite dimensional primal problem \eqref{eq:primalmadefinite} to construct $\hat u$.
\end{enumerate}
Each step within this algorithmic framework however suffers from serious issues:
\begin{description}
 \item[Problem 1] the dual problem \eqref{eq:dualgood} is finite dimensional but involves two infinite dimensional convex constraints sets
\begin{equation}
\Q_1= \{q\in \R^m, A^*q\in \ran L^*\} 
\end{equation}
and 
\begin{equation}
\Q_2= \{q\in \R^m,  \norm{(L^+)^*A^*q}_\infty \leq 1\},
\end{equation}
which need to be handled with a computer.
\item[Problem 2] finding $I(\hat q)$ again consists of a possibly nontrivial maximization problem.
\item[Problem 3] the set $I(\hat q)$ may not be finitely supported. 
\end{description}

To the best of our knowledge, finding general conditions on the functions 
\begin{equation}
 \rho_i = (L^+)^*a_i
\end{equation}
allowing to overcome those hurdles is an open problem. It is however known that certain family of functions including polynomials and trigonometric polynomials \cite{lasserre2001global} allow for a numerical resolution. 
In the following two sections, we study two specific cases useful for applications in details: the piecewise linear functions and trigonometric polynomials.

\subsubsection{Piecewise linear functions in arbitrary dimensions} \label{sec:PW}

In this section, we assume that $\Omega$ is a bounded polyhedral subset of $\R^d$ and that each $\rho_i=(L^+)^*a_i$ is a piecewise linear function, with finitely many regions, all being polyhedral.
This class of functions is commonly used in the finite element method. 
Its interest lies in the fact that any smooth function can be approximated with an arbitrary precision by using mesh refinements.

\paragraph{Solving the primal}
For this class, notice that the function $(L^+)^*A^*q=\sum_{i=1}^m q_i \rho_i$ is still a piecewise linear function with finitely many polyhedral pieces. The maximum of the function has to be attained in at least one of the finitely many vertices $(v_j)_{j\in J}$ of the pieces. 
This is a key observation from a numerical viewpoint since it simultaneously allows to resolve problems 1 and 2.  
First, the constraint set $\Q_2$ can be described by a finite set of linear inequalities:
\begin{align*}
	-1 \leq (L^+)^*A^*q(v_j) \leq 1, \quad j \in J.
\end{align*} 
Secondly, $I(\hat{q})$  can be retrieved by evaluating $(L^+)^*A^*q$ only on the vertices $(v_j)_{j \in J}$. 

Unfortunately, problem 3 is particularly important for this class: $I(\hat{q})$ needs not be finitely supported since the maximum could be attained on a whole face. The following proposition however confirms that there always exists solutions supported on the vertices.
\begin{proposition} \label{prop:PiecewiseLinear}
Suppose that $\ran L = \M$ and that the dual problem \eqref{eq:dualgood} has a solution. Then Problem \eqref{eq:mainproblem} has at least one solution of the form
\begin{equation}\label{eq:solutionpiecewiselinear}
\hat u = \sum_{j\in J} d_j L^+\delta_{v_j} +u_K
\end{equation}
with $u_K \in \ker L$, $d_j \in \R$ and the $v_j$ are the vertices of the polyhedral pieces. 
\end{proposition}
Once again, knowing the locations $v_j$ of the Dirac masses in advance permits to solve \eqref{eq:primalmadefinite} directly (without solving the dual) in order to obtain an exact solution of \eqref{eq:mainproblem}.

\paragraph{Sparsifying the solution}
\label{sec:sparsification}

For piecewise linear measurement functions, it turns out that the solution is not unique in general and that the form \eqref{eq:solutionpiecewiselinear} is not necessarily the sparsest one. A related observation was already formulated in a different setting in \cite{duval2015exact}, where the authors show that in 1D, two Dirac masses are usually found when only one should be detected. 
Figure \ref{fig:piecewiselinear} illustrates different types of possible solutions for a 2D mesh. 
\begin{figure}
\centering
\includegraphics[width=.75\textwidth]{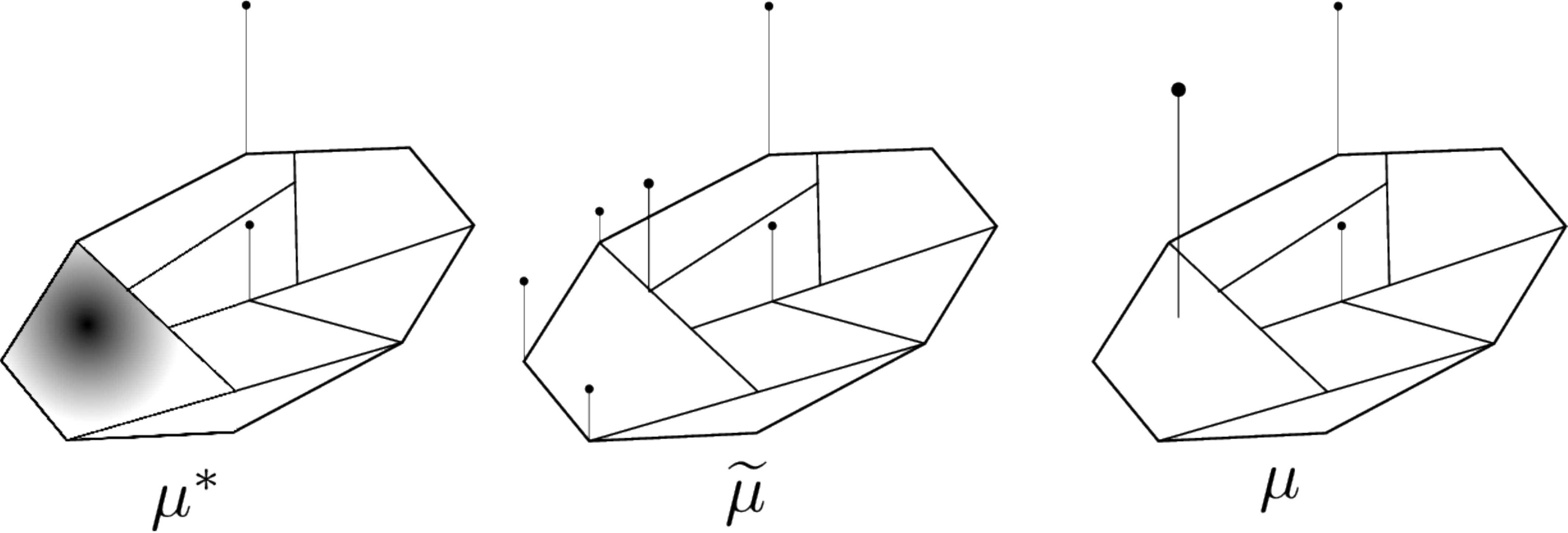}
\caption{A graphical depiction of the three types of solutions for piecewise linear measurements. \label{fig:piecewiselinear}}
\end{figure}

The proof of Proposition \ref{prop:PiecewiseLinear} suggests that one can sparsify a solution found by solving the primal problem resulting from the discretization through sampling on the grid of vertices. 
The basic reason is that piecewise linear measurements specify the zero-th and first order moments of a measure restricted to one piece. Among the infinitely many measures having these moments, one can pick the sparsest one, consisting of a unique Dirac mass. 
This principle allows to pass from the 6-sparse measure $\tilde \mu$ to the 3-sparse measure $\mu$ in Figure \ref{fig:piecewiselinear}. 

To be precise, a collection of peaks $(d_i \delta_{x_i})_{i \in I}$, where
\begin{enumerate}
\item $\mathrm{conv} (x_i)_{i \in I}$ is contained in one polyhedral region of $\cal G$.
\item $(d_i)_{i \in I}$ have the same sign $\epsilon \in \set{-1,1}$
\end{enumerate}
can be combined into one peak $\overline{d} \delta_{\overline{x}}$, with
\begin{align*}
	\overline{d} = \sum_{i \in I} d_i, \quad \overline{x} = \frac{1}{\overline{d}} \sum_{i \in I} d_i x_i.
\end{align*}
We will see in the numerical experiments that this seemingly novel principle allows exact recovery of $u$ under certain conditions on its initial structure.


\paragraph{Relationship to standard discretization}

The traditional way to discretize total variation problems with $L=\Id$ consists in imposing the locations of the Dirac masses on a set of predefined points $(x_i)_{1\leq i \leq n}\in \Omega^n$. Then, one can look for a solution of the form $\hat u = \sum_{i=1}^n d_k \delta_{x_i}$ and inject this structure in problem \ref{eq:mainproblem}. 
By using this reasoning, there is no reason to find the exact solution of the original infinite dimensional problem.
Proposition \eqref{prop:PiecewiseLinear} sheds a new light on this strategy, by telling that this actually amounts to solving \emph{exactly} an infinite dimensional problem with piecewise linear measurement functions. 

\subsubsection{Trigonometric polynomials in 1D}

In this section, we assume that $\Omega=\mathbb{T}$ is the one dimensional torus (see remark \eqref{rem:Omega}). For $j\in \N$, let $p_j(t)=\exp(-2 \iota \pi j t)$. 
We also assume that the functions $\rho_i$ are real trigonometric polynomials:
\begin{equation*}
 \rho_i = \sum_{j=-K}^K \gamma_{j,i} p_j,
\end{equation*}
with $\gamma_{j,i} = \gamma_{-j,i}^*$.
For this problem, the strategy suggested in section \ref{sec:dual} will be adopted.

\paragraph{Solving the dual}
The following simple lemma states that in the case of a finite dimensional kernel, the constraint set $\Q_1$ is just a finite dimensional linear constraint.
\begin{lemma}\label{lem:range}
Let $r=\dim(\ker(L))<\infty$ and $(\lambda_i)_{1\leq i\leq r}$ denote a basis of $\ker L$. 
The set $\Q_1$ can be rewritten as
\begin{equation*}
 \Q_1=\{q\in \R^m, \forall 1\leq i \leq r, \langle q, A \lambda_i\rangle =0\}.
\end{equation*}
\end{lemma}
\begin{proof}
Since $\ran L^*=V$ (by the closed range theorem), $A^*q\in \ran L^*$ if and only if $\forall 1\leq i\leq r, \langle A^*q,\lambda_i\rangle =0$.
\end{proof}
Hence, when $\ker L$ is finite-dimensional the set $\Q_1$ can be easily handled by using numerical integration procedures to compute the $mr$ scalars $\langle a_k, \lambda_i\rangle$. Let us now turn to the set $\Q_2$. The following lemma is a simple variation of \cite[Thm 4.24]{dumitrescu2007positive}. It was used already for super-resolution purposes \cite{candes2014towards}. 
\begin{lemma}\label{lem:trigo}
 The set $\Q_2$ can be rewritten as follows:
 \begin{multline*}
  \Q_2=\Bigg\{ \alpha \in \R^m, \exists Q\in \C^{(2K+1)\times (2K+1)},\begin{bmatrix}
                                                             Q & \Gamma \alpha \\
                                                             (\Gamma\alpha)^* & 1
                                                            \end{bmatrix}\succeq 0,\\
\sum_{i=1}^{2K+2-j}Q_{i,i+j}=
\begin{cases}
1, & j=0, \\
0, & 1\leq j \leq 2K+1.                                                                                                \end{cases}\Bigg\}.
\end{multline*}                             
\end{lemma}
With Lemmas \ref{lem:range} and \ref{lem:trigo} at hand, the dual problem \eqref{eq:dualgood} becomes a semidefinite program that can be solved with a variety of approaches, such as interior point methods \cite{vandenberghe1996semidefinite}.

\paragraph{Finding the Dirac mass locations}

The case of trigonometric polynomials makes Proposition \ref{prop:dualtoprimal} particularly helpful. In that case, either the trigonometric polynomial is zero and the solution $\hat u$ lives in the kernel of $L$, or the set $I$ is finite with cardinality at most $2K$, since $|(L^+)^*A^*q|^2-1$ is a negative trigonometric polynomial of degree $4K+2$.
Retrieving its roots can be expressed as an eigenvalue evaluation problem \cite{chandrasekaran2007fast} and be solved efficiently.

\section{Numerical Experiments}\label{sec:numerical}
In this section, we perform a few numerical experiments to illustrate the theory.
In all our experiments, we use the toolbox CVX \cite{cvx} for solving the resulting convex minimization problems.

\subsection{Piecewise linear functions} \label{sec:NumPW}

\subsubsection{Identity in 1D}

In this paragraph, we set $L=\Id$ and $\Omega=[0,1]$. We assume that the functions $a_i$ are random piecewise linear functions on a regular grid. The values of the functions on the vertices are taken as independent random Gaussian realizations with standard deviation $1$. In this experiment, we set $u$ as a sparse measure supported on 3 points. We probe it using 12 random measurement functions $a_i$ and do not perturb the resulting measurement vector $b$, allowing to set $f_b=\iota_{\{b\}}$.
The result is shown on Figure \ref{fig:1Dlinear}. As can be seen, the initially recovered measure is $7$ sparse. Using the sparsification procedure detailed in paragraph \ref{sec:sparsification} allows to \emph{exactly} recover the true $4$ sparse measure $u$. 
We will provide a detailed analysis of this phenomenon in a forthcoming paper.

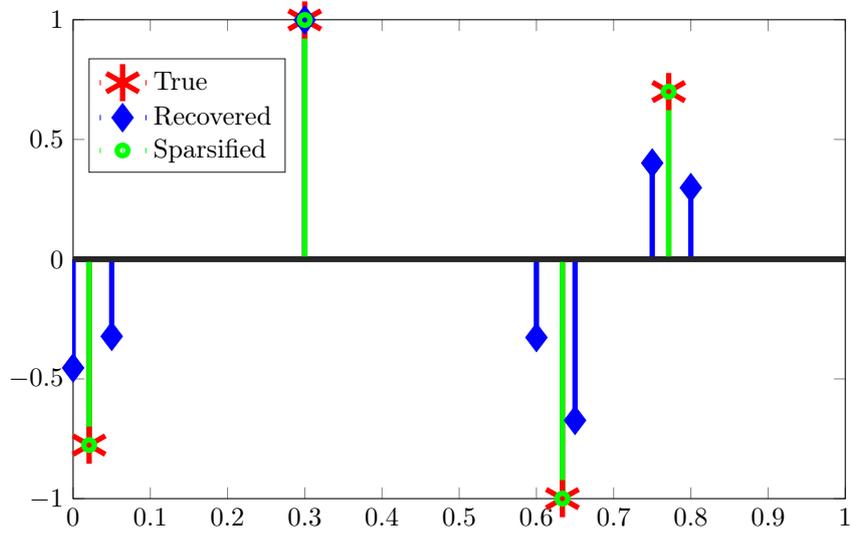
\begin{figure}[htp]
	\centering
%
%
\begin{tikzpicture}

\begin{axis}[%
width=4in,
height=2.5in,
at={(0.791in,0.502in)},
scale only axis,
xmin=0,
xmax=1,
xtick={   0, 0.1, 0.2, 0.3, 0.4, 0.5, 0.6, 0.7, 0.8, 0.9, 1},
ymin=-1,
ymax=1,
axis background/.style={fill=white},
legend style={legend cell align=left, align=right, draw=white!15!black,at={(0.02,0.8)},anchor=west}
]
\addplot[ycomb, color=red, line width=2.0pt, mark size=7.0pt, mark=asterisk, mark options={solid, red}] table[row sep=crcr] {%
0.0207519493594015	-0.776337745438523\\
0.3	0.998991790717395\\
0.633648234926275	-1\\
0.771320643266746	0.7\\
};
\addplot[forget plot, color=white!15!black, line width=2.0pt] table[row sep=crcr] {%
0	0\\
1	0\\
};
\addlegendentry{True}

\addplot[ycomb, color=blue, line width=2.0pt, mark size=4.0pt,mark=diamond*, mark options={solid, blue}] table[row sep=crcr] {%
0	-0.454127111857697\\
0.05	-0.322210498671331\\
0.3	0.998991783932209\\
0.6	-0.32703523016219\\
0.65	-0.67296473719843\\
0.75	0.401510857397635\\
0.8	0.298488892006686\\
};
\addplot[forget plot, color=white!15!black, line width=2.0pt] table[row sep=crcr] {%
0	0\\
1	0\\
};
\addlegendentry{Recovered}

\addplot[ycomb, color=green, line width=2.0pt, mark=o, mark options={solid, green}] table[row sep=crcr] {%
0.0207519572864545	-0.776337610529028\\
0.3	0.998991783932209\\
0.633648237958179	-0.999999967360619\\
0.771320642775993	0.699999749404321\\
};
\addplot[forget plot, color=white!15!black, line width=2.0pt] table[row sep=crcr] {%
0	0\\
1	0\\
};
\addlegendentry{Sparsified}

\end{axis}
\end{tikzpicture}%
\caption[Recovery with random piecewise linear measurements in 1D]{Example of recovery with random piecewise linear measurement functions in 1D. The solution recovered by a standard $\ell^1$ solver is not the sparsest one. The sparsification procedure proposed in the paper allows recovering the sparsest solution and recovering \emph{exactly} the sampled function. \label{fig:1Dlinear}}
\end{figure}

\subsubsection{Derivative in 1D}

In this section we set $\Omega=[0,1]$ and $L=D$. We assume that the functions $a_i$ are piecewise constant. In the terminology of \cite{unser2016splines}, this means that we are \emph{sampling splines with splines}. By equation \eqref{eq:defDplusstar}, we see that the functions $\rho_i=(L^+)^*a_i$ are piecewise linear and satisfy $\rho_i(0)=\rho_i(1)=0$. 

In this example, we set the values of $a_i$ on each piece as the realization of independent normally distributed random variables. We divide the interval $[0,1]$ in $10$ intervals of identical length. 
An example of a sampling function is displayed in Figure \ref{fig:1DDSampling}.
\begin{figure}[htp]
	\centering
\input{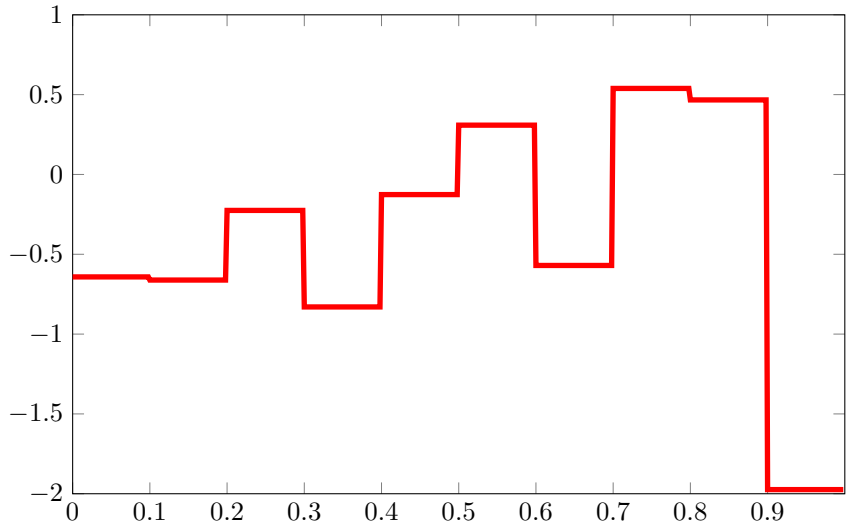}
\caption[Sampling function to probe piecewise constant signals]{Sampling function $a_1$ used to probe piecewise constant signals. The others have a similar structure with other random values on each interval. \label{fig:1DDSampling}}
\end{figure}

The sensed signal $u$ is defined as piecewise constant with jumps occurring outside the grid points. Its values are comprised in $[-1,1]$.

The measurements are obtained according to the following model: $b_i=\langle a_i,u\rangle + \epsilon_i$, where $\epsilon_i$ is the realization of a Bernoulli-Gaussian variable. It takes the value $0$ with probability $0.9$ and takes a random Gaussian value with variance $3$ with probability $0.1$. To cope with the fact that the noise is impulsive, we propose to solve the following problem $\ell^1$ fitted and total variation regularized problem.
\begin{equation}
 \min_{u\in BV(]0,1[)}  \|Du\|_{TV} + \alpha\|Au - b\|_1,
\end{equation}
where $\alpha=1$.

A typical result of the proposed algorithms is shown in Figure \ref{fig:1DDlinear}. 
Here, we probe a piecewise constant signal with 3 jumps (there is a small one in the central plateau) with 42 measurements. Once again, we observe perfect recovery despite the additive noise. This favorable behavior can be explained by the fact that the noise is impulsive and by the choice of an $\ell^1$ data fitting term.

\begin{figure}[htp]
	\centering
\input{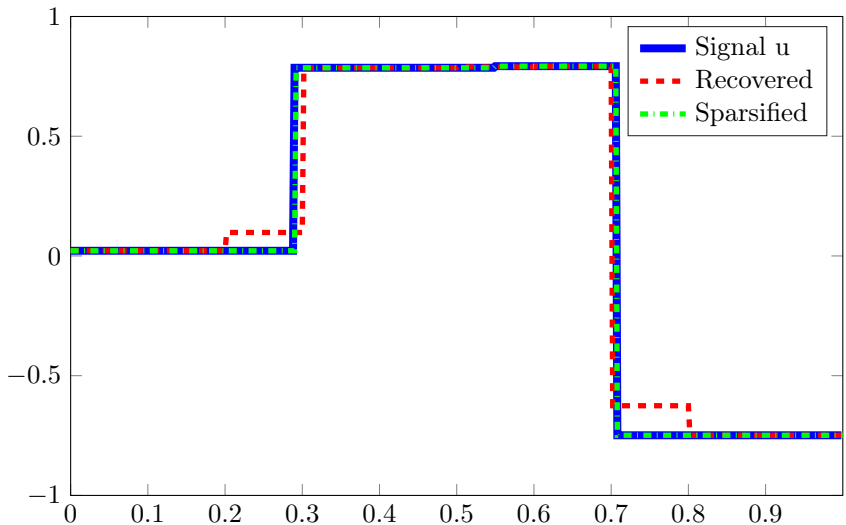}
\caption[Recovery of a piecewise linear signal]{Example of recovery of a piecewise linear signal $u$ with measurements corrupted by Bernoulli-Gaussian noise. Once again, the proposed algorithm implemented with the proposed sparsification procedure recovers the true signal exactly, despite noise.\label{fig:1DDlinear}}
\end{figure}

\subsubsection{Identity in 2D}

In this section, we set $\Omega=[0,1]^2$ and $L=\Id$.
We probe a sparse measure $\mu \in \calM([0,1]^2)$ using real trigonometric polynomials up to order $5$. We then solve the problem \ref{eq:mainproblem} with $f_b$ modeling box-constraints and $A$ being the measurement operator associated with the piecewise functions formed by linearizing the trigonometric polynomials on a regular grid $\set{0,0.1, \dots, 1}^2$. Then, we collapse the resulting solution into a sparser one. To avoid numerical problems, we discarded all peaks with an amplitude less than $10^{-8}$ before the last step.  The results, together with an illustration of the collapsing procedure, are depicted in Figure \ref{fig:Sparsification}.

\begin{figure}[htp] \label{fig:Sparsification}
\centering
%
%
\definecolor{mycolor1}{rgb}{0.00000,0.44700,0.74100}%
\begin{tikzpicture}

\begin{axis}[%
width=1.202in,
height=1.5in,
at={(0.777in,0.368in)},
scale only axis,
xmin=0,
xmax=1,
tick align=outside,
ymin=0,
ymax=1,
zmin=-1,
zmax=1,
view={-27.1}{25.9999999999999},
axis background/.style={fill=white},
axis x line*=bottom,
axis y line*=left,
axis z line*=left,
xmajorgrids,
ymajorgrids,
zmajorgrids
]
\addplot3 [ycomb, color=mycolor1, mark size=2.5pt, mark=*, mark options={solid, fill=mycolor1, mycolor1}, forget plot]
 table[row sep=crcr] {%
0.244459115702494	0.502552972261186	1\\
0.761481386819517	0.249092308285744	-1\\
0.739207782752857	0.748966286054241	1\\
};
 \end{axis}

\begin{axis}[%
width=1.202in,
height=1.5in,
at={(2.456in,0.368in)},
scale only axis,
xmin=0,
xmax=1,
tick align=outside,
ymin=0,
ymax=1,
zmin=-1,
zmax=1,
view={-33.2000000000001}{24.4},
axis background/.style={fill=white},
axis x line*=bottom,
axis y line*=left,
axis z line*=left,
xmajorgrids,
ymajorgrids,
zmajorgrids
]
\addplot3 [ycomb, color=mycolor1, mark size=2.5pt, mark=*, mark options={solid, fill=mycolor1, mycolor1}, forget plot]
 table[row sep=crcr] {%
0	0.666666666666667	0.010626143619589\\
0	0.733333333333333	1.37523026905238e-08\\
0.133333333333333	0.466666666666667	-0.0450921242250679\\
0.2	0.533333333333333	0.236246677354495\\
0.2	0.6	0.024808938992141\\
0.266666666666667	0.266666666666667	0.0542900706448661\\
0.266666666666667	0.466666666666667	0.571592954755445\\
0.266666666666667	0.533333333333333	0.0102937080194551\\
0.333333333333333	0.333333333333333	-0.0464795924861911\\
0.466666666666667	0.466666666666667	0.0296791906479412\\
0.533333333333333	0.266666666666667	-0.0189319739687467\\
0.533333333333333	0.333333333333333	-0.0348865609697112\\
0.666666666666667	0.0666666666666667	0.0109966236047616\\
0.666666666666667	0.266666666666667	0.0162055760941564\\
0.666666666666667	0.466666666666667	0.00893632858406185\\
0.733333333333333	0.266666666666667	-0.522218100589554\\
0.733333333333333	0.333333333333333	-0.100450842317402\\
0.733333333333333	0.6	-0.0955986463855595\\
0.733333333333333	0.733333333333333	0.528459052512244\\
0.733333333333333	0.8	0.312657560253134\\
0.8	0.2	-0.350738926987296\\
0.8	0.666666666666667	0.0610009275187363\\
0.8	0.733333333333333	0.0368420496173009\\
0.866666666666667	0.0666666666666667	0.0101298265793082\\
0.866666666666667	0.333333333333333	0.0439117533420127\\
0.866666666666667	0.4	0.00251637188743046\\
0.933333333333333	0.466666666666667	-0.00117712724707843\\
1	0.666666666666667	0.0106261423957693\\
1	0.733333333333333	1.37523026905255e-08\\
};
 \end{axis}

\begin{axis}[%
width=1.202in,
height=1.5in,
at={(4.135in,0.368in)},
scale only axis,
xmin=0,
xmax=1,
tick align=outside,
ymin=0,
ymax=1,
zmin=-1,
zmax=1,
view={-34.8}{25.2000000000001},
axis background/.style={fill=white},
axis x line*=bottom,
axis y line*=left,
axis z line*=left,
xmajorgrids,
ymajorgrids,
zmajorgrids
]
\addplot3 [ycomb, color=mycolor1, mark size=2.5pt, mark=*, mark options={solid, fill=mycolor1, mycolor1}, forget plot]
 table[row sep=crcr] {%
0.247415796768358	0.486756332843817	0.818133340129395\\
0.736130887189101	0.757074584764968	0.877958662382679\\
0.760118838683341	0.239881161316659	-0.872957027576849\\
0.733333333333333	0.333333333333333	-0.100450842317402\\
0.733333333333333	0.6	-0.0955986463855595\\
0.8	0.666666666666667	0.0610009275187363\\
0.266666666666667	0.266666666666667	0.0542900706448661\\
0.333333333333333	0.333333333333333	-0.0464795924861911\\
0.133333333333333	0.466666666666667	-0.0450921242250679\\
0.866666666666667	0.336946619999544	0.0464281252294432\\
0.533333333333333	0.309881717651864	-0.0538185349384579\\
0.466666666666667	0.466666666666667	0.0296791906479412\\
0.2	0.6	0.024808938992141\\
0.666666666666667	0.266666666666667	0.0162055760941564\\
0.666666666666667	0.0666666666666667	0.0109966236047616\\
0	0.666666752946226	0.0106261573718917\\
1	0.666666752946236	0.010626156148072\\
0.866666666666667	0.0666666666666667	0.0101298265793082\\
0.666666666666667	0.466666666666667	0.00893632858406185\\
0.933333333333333	0.466666666666667	-0.00117712724707843\\
};
 \end{axis}
\end{tikzpicture}%

\includegraphics[width=\textwidth]{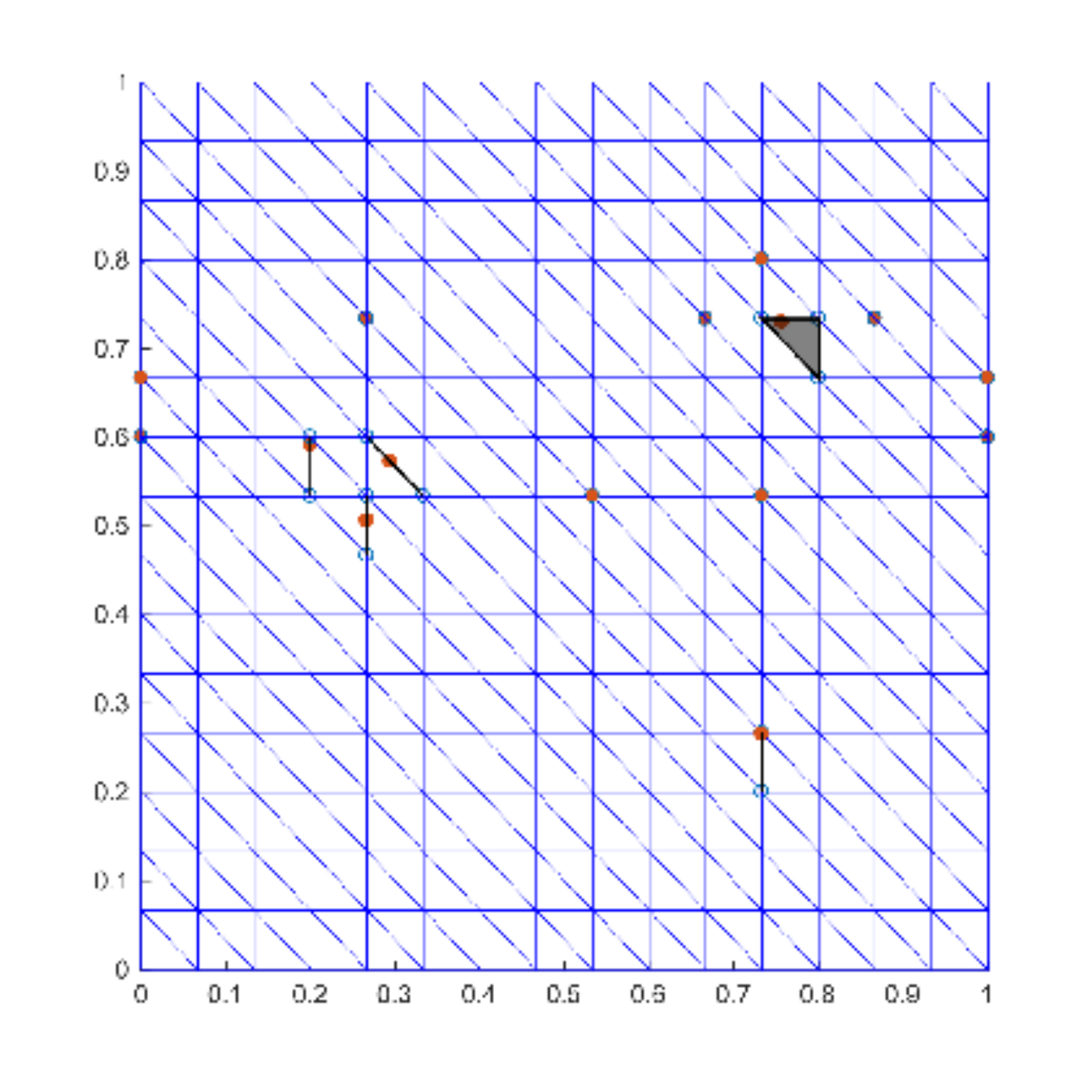}
\caption[Recovery with piecewise linear measurements in 2D]{Top: The ground truth measure, the solution obtained by sampling on the vertices, and the sparsified solution. Bottom: Illustration of the sparsification procedure. The circles  represent the initial solution, while the dots indicate the sparsified solution. A thick trait or a grey region indicates masses that have been merged.}
\end{figure}
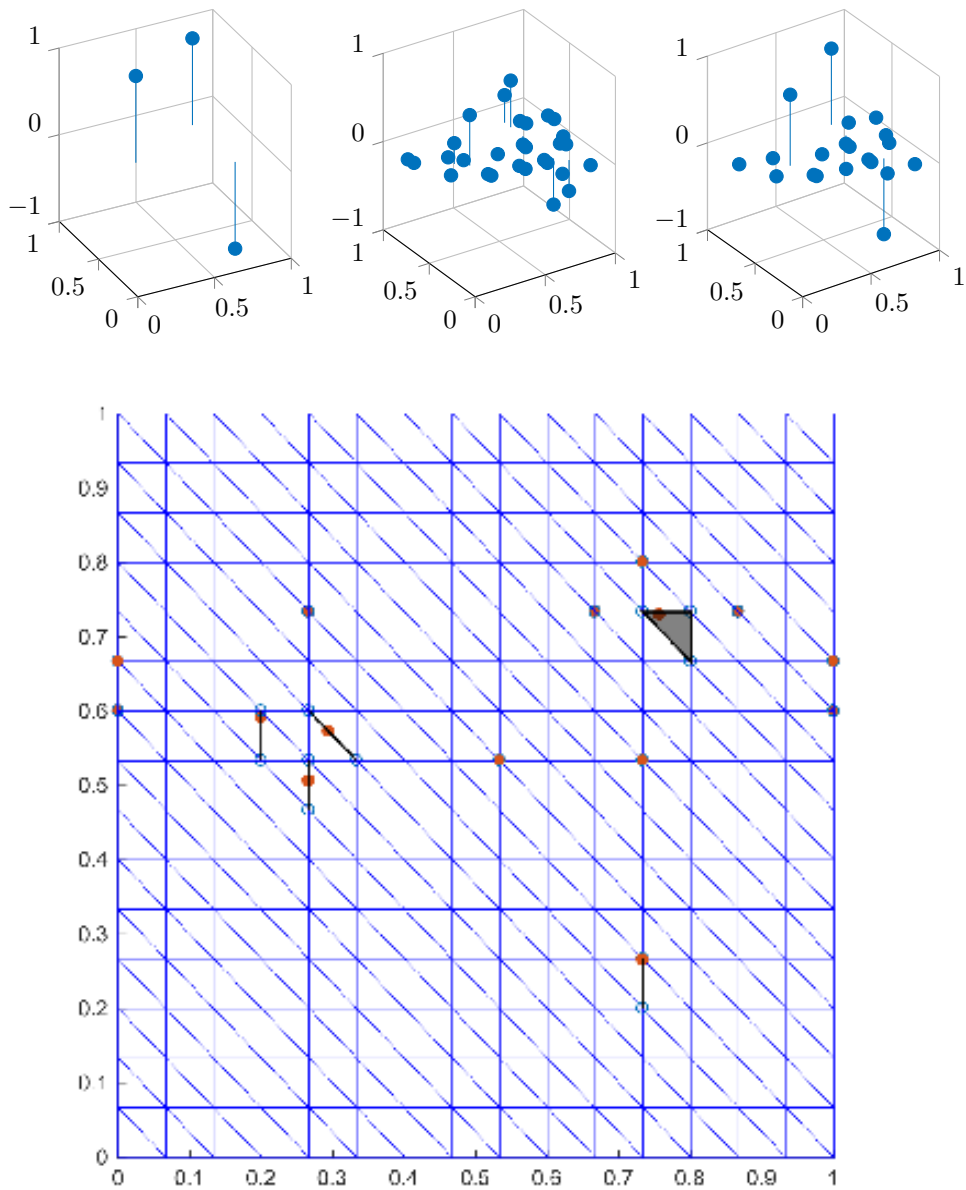

\subsection{Trigonometric polynomials.}
We generate $m=35$ trigonometric polynomials $$a_i(t) = \sum_{j=-N}^N \gamma_{j,i} \exp(-2\iota \pi j t)$$ of degree $N=50$ as follows: for $j\geq 0$, we set the coefficients $\gamma_{j,i}$ of the $i$:th polynomial to be
\begin{align*}
	\gamma_{j,i} = \frac{\xi_{j,i}}{\max(j,1)},
\end{align*}
where $\xi_{j,i}$ are i.i.d. normal distributed. For $j<0$, we set $\gamma_{j,i} = \gamma^*_{-j,i}$. This ensures that the functions $a_i$ are real, and furthermore have a good approximation rate with respect to trigonometric polynomials. Seven such functions are depicted in Figure\ref{fig:MeasurementFunctions}. Note that we do not need to worry about $a_i$ not vanishing at $\pm 1/2$, since the functions live on the torus, a manifold without boundary.

\begin{figure}
\center
\includegraphics[width=12cm]{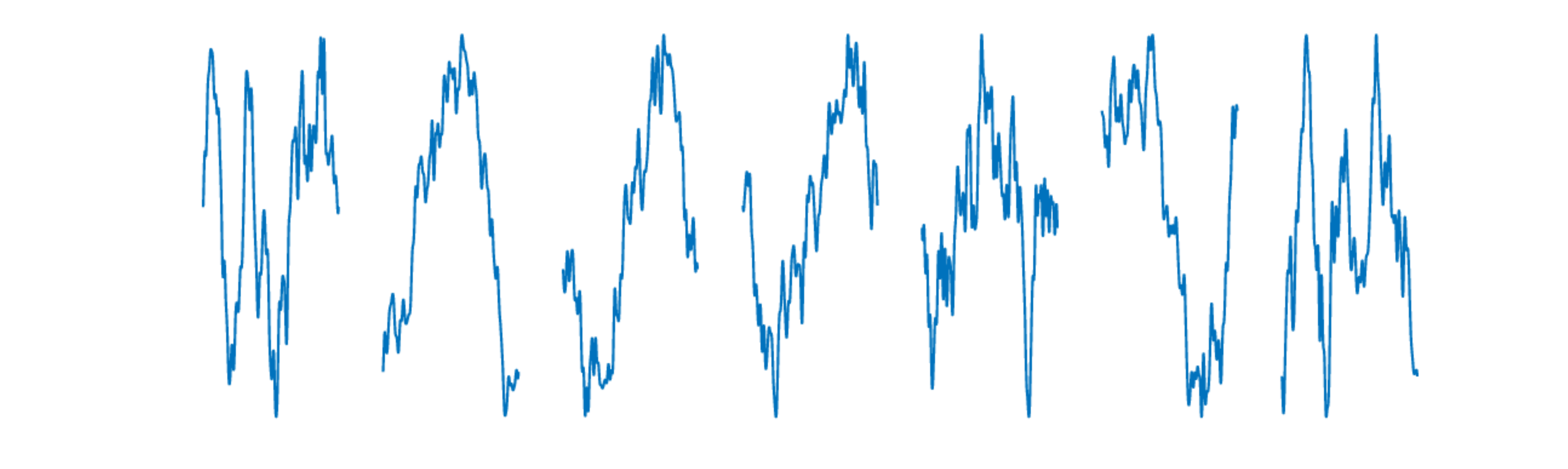}
\caption[Randomly generated trigonometric polynomials]{Seven randomly generated trigonometric polynomials $a_i$.\label{fig:MeasurementFunctions}}
\end{figure}

 We then generate $b \in \R^{m}$ by measuring a ground truth measure $\mu_0 = \sum_{i=-2}^2 c_i \delta_{x_i}$, where $x_i$ are chosen as 
 $$ x_i = \frac{i}{5} + n_i, $$ where $n_i$ are small random displacements, and $i.i.d$ normally distributed amplitudes $(c_i)_{i=1}^5$. Next, for each $K=10,11, \dots, 50$, we solve the problem \ref{eq:mainproblem}, with $A$ being the measurement operator with respect to the functions $$\widetilde{a}_i^K(t) = \sum_{j=-K}^K \gamma_{j,i} \exp(-2\iota \pi j t).$$
 
 In Figure \ref{fig:trig}, we plot the results of the minimization \eqref{eq:mainproblem} with $$f_b(x) = \frac{100 \norm{x}_2^2}{2},$$ depending on $K$. We see that already for $K =30$, the solution is reasonably close to the true solution (at $N=50$) (the relative error in the input, $\Vert{\widetilde{A}\mu_0 -b}\Vert_2 / \norm{b}_2$, for this $K$ is approximately equal to $0.06$). The latter is furthermore essentially equal to the ground truth $\mu_0$.
 
 \begin{figure} [htp]
 \centering
 \input{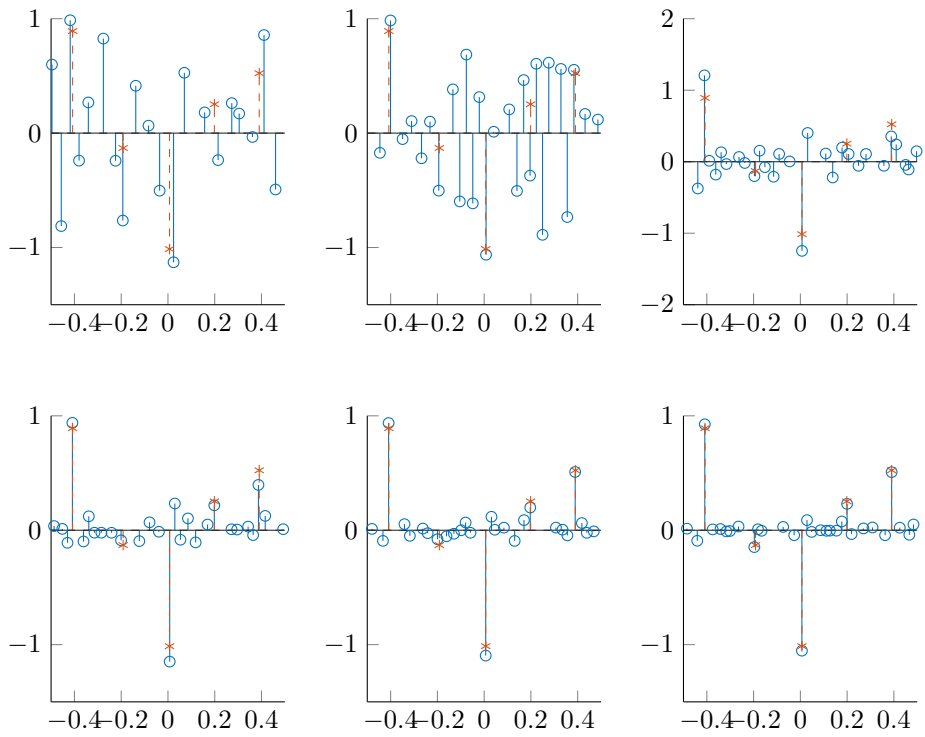}
  \caption[Recovery with trigonometric measurements]{Minimizers of \eqref{eq:mainproblem} ($*$) together with the ground truth $\mu_0$ ($\circ$) for (from above left to below right) $K$ equal to $15,20,25,30,35$ and $40$, respectively. \label{fig:trig}}
 \end{figure}

\section{Proofs}\label{sec:proofs}
In this section, we include all proofs left out in the main text.
\subsection{Structure of solutions}
As was argued already in the main body of the text, the proof of Theorem \ref{thm:structure} can be broken down to a treatment of the problem \eqref{eq:reducedProblem}. In the following, we will carry out the argument proving that the latter problem has a solution of the claimed form.

We first prove the result in finite dimensions and then use a limit argument.
The statement is well known in finite dimension, see e.g. \cite[Theorem 6]{unser2016representer} and \cite{tibshirani1996regression}. We provide a proof for completeness. It has a geometrical flavour.

\begin{lemma} \label{lem:finiteDimensionalStructure}
	Let $m,n \in \N$,  $G\in \R^{m,n}$, $b \in \ran G$ and $m\leq n$. Then a problem of the form
	\begin{align}
		\min_{u \in \R^n} \norm{u}_1  \st G u =b \label{eq:finDimProb}.
    \end{align}	 
has a solution $\hat{u}$ of \eqref{eq:mainproblem} of the form
	\begin{align*}
		\hat{u} = \sum_{k=1}^p c_k e_{i_k},
	\end{align*}
	with $(c_k)_{k=1}^m$ some real scalars and $p\leq m$.
\end{lemma}

\begin{proof}
	
	Let $u$ be a solution to \eqref{eq:finDimProb} (its existence easily follows from the coercivity of the $1$-norm and the non-emptiness and closedness of the set $G^{-1}(\set{b})$).  The image $b =G u$ then lies on the boundary of the polytope $P=G \set{ u \, \vert \, \norm{u}_1 \leq \norm{u}_1}$ -- if it did not, $b$ would be of the form $G\tilde{u}$ with $\norm{\tilde{u}}_1 < \norm{u}_1$. Then $\tilde{u}$ would be a feasible point with smaller objective value than $u$, which is a contradiction to the optimality of $u$.
	
	The polytope $P$ is at most $m$-dimensional, hence its boundary $\partial P$ consists of faces of dimension at most $m-1$. Having just argued that $b$ lies on that boundary, it must lie on one of those faces, say $F$, which then has dimension at most $m-1$. Concretely, $b \in \mathrm{conv}(\mathrm{vert}(F))$, where $\mathrm{vert}(F)$ denotes the set of vertices of face $F$. The vertices of $F$ are the images by $G$ of a subset of the $\ell_1$-ball's vertices, so they can be written as $\norm{u}_1 \epsilon_k Ge_{i}$, for some $i\in \{1,\hdots, n\}$ and for $\epsilon_k\in \{-1,1\}$.
	Caratheodory's theorem applied in the $(m-1)$-dimensional space $\mathrm{aff } F$ implies that $b$ can be written as
	\begin{align*}
		b=\sum_{k=1}^m  \theta_k \norm{\hat{u}}_1 \epsilon_k Ge_{i_k}
	\end{align*}
	with $\sum_{k=1}^m \theta_k=1$ and $\epsilon \in \set{\pm 1}^m$.  The vector $\norm{u}_1\sum_{k=1}^m  \theta_k \epsilon_k e_{i_k}$ is a  solution of \eqref{eq:finDimProb} of the stated form.
\end{proof}

The strategy will now be to discretize the problem on finer and finer grids, use the previous lemma  and pass to the limit. 

\begin{lemma}\label{lem:discretization}
 Define a sequence of discretizations $(\Omega_n)_{n\in \N}$ of $\Omega$ as
    \begin{equation}
     \Omega_n = \left([-2^n,2^n]^d\cap \frac{\mathbb{Z}^d}{2^n}\right) \cap \Omega.
    \end{equation}
    For $k\in \Omega_n$, define $\omega_n^k$ to be the hypercube of center $k$ and side-length $2^{-n}$ intersected with $\Omega$.
    Let $\mu\in \M(\Omega)$ denote a measure and define the sequence:
    \begin{equation}
      \nu_n=\sum_{k\in \Omega_n} \mu(\omega_k) \delta_k.
    \end{equation}
    
    Then $\nu_n \wstarto \mu$ and $\|\nu_n\|_{TV}\leq \|\mu\|_{TV}$.
\end{lemma}
\begin{proof}
	First, it follows directly from the definition of the total variation that
	\begin{align}
	\norm{\nu_n}_{TV} = \sum_{k \in \Omega_n} \abs{\mu(\omega_k)} \leq \norm{\mu}_{TV}.
	\end{align}
	We now need to prove that for each $\phi \in \calM^*$, $\sprod{\nu_n, \phi} \to \sprod{\mu, \phi}$. So fix $\phi$ and let $\epsilon>0$. Since $\phi \in \M^*$, there exists a compact set $K$ with the property $\abs{\phi(x)} < \epsilon$ for $x \notin K$.  Since $\phi$ is equicontinuous on $K$, there exists a $\delta>0$ so that if $\norm{x-x'}_\infty<\delta,$ $\abs{\phi(x)-\phi(x')}<\epsilon$. If we choose $n$ so large so that $2^{-n}<\delta$, we will have
	\begin{align*}
		\abs{\sprod{\mu - \nu_n, \phi}} &\leq \int_{\Omega \backslash K} \abs{\phi}d(\abs{\mu}+\abs{\nu_n}) + \abs{ \int_K \phi d\mu - \int_K \phi d \nu_n} \\
		&\leq  \epsilon ( \norm{\mu}_{TV} + \norm{\nu_n}_{TV}) + \abs{ \sum_{k \in \Omega_n} \int_{\omega_k} \phi d\mu - \phi(k) \mu(\omega_k)} \\
		& \leq 2\epsilon \norm{\mu}_{TV} +  \sum_{k \in \Omega_n} \int_{\omega_k} \abs{\phi(\ell)-\phi(k)} d\mu(\ell) \\
		&\leq 2\epsilon \norm{\mu}_{TV} + \epsilon \sum_{k \in \Omega_n}  \abs{\mu(\omega_k)} \\
		&\leq 3 \epsilon \norm{\mu}_{TV}.
	\end{align*}
	Since $\epsilon>0$ was arbitrary, the claim follows.
\end{proof}

When passing to the limit in our limit argument, we we will need the following continuity property of the operator $AL^+$:
\begin{lemma} \label{lem:weakStarWeak}
	The operator $AL^+: \calM \to \R^m$ is weak-$*$-weak continuous. That is, if $\mu_n \wstarto \mu$, $AL^+\mu_n \to AL^+\mu$. The same is true for $H=\Pi_{X^\perp}A$.
\end{lemma}

\begin{proof}
We simply need to note that $\mu_n \wstarto \hat{\mu}$ and assumption \ref{assump:a} implies that
	\begin{equation}
	\sprod{a_i, L^+\mu_n} = \sprod{(L^+)^*a_i, \mu_n} = \sprod{\mu_n , \rho_i} \to \sprod{\mu, \rho_i} = \sprod{a_i, L^+\mu} = \sprod{(L^+)^*a_i, \mu},
\end{equation}
and that $\Pi_{X^\perp}$ is continuous.
\end{proof}

Now let us prove that the optimal value of the problem \eqref{eq:mainproblem} can be found by solving slightly perturbed discretized problems.
\begin{lemma} \label{lem:rightDiscrete}
	Let $b \in \ran H$. There exists a sequence $(b_n)_{n\in \N}$ of vectors in $\R^m$ with the following properties
\begin{itemize}	
\item For each $n$, $b_n$ is in the range of the $n$-th discretized $H$-operator, i.e.
\begin{align*}
	b_n \in H  \vect\left( \delta_{\omega}\right)_{\omega \in \Omega_n}
\end{align*}
\item $b_n$ converges to $b$.
	\item For $n \in \N$, define $\hat{J}_n$ through 
	\begin{align}\label{eq:Pn}
		\hat J_n := \min_{c \in \R^{|\Omega_n|}} \norm{c}_1 \st  H L^+\left(\sum_{k \in \Omega_n}c_k \delta_{k}\right)=b_n. \tag{$\calP_n$}
	\end{align}
Then $\liminf_{n \to \infty} \hat{J}_n \leq \hat J$, where $\hat J$ is the optimal value of problem \eqref{eq:reducedProblem}.
	\end{itemize}
\end{lemma}

\begin{proof}
	First, we note that problem \eqref{eq:reducedProblem} has a solution $\hat{\mu}$. We skip the proof since it is identical to that of Proposition \ref{prop:solvability}.
	
	 Now, according to Lemma \ref{lem:discretization}, there exists a sequence of measures $\mu_n$ of the form
	\begin{align*}
		\mu_n=\sum_{k \in \Omega_n}c_k \delta_{k}
	\end{align*}
	with $\mu_n\wstarto \hat{\mu}$ and $\norm{\mu_n}_{TV} \leq \norm{\hat \mu}_{TV}$ for each $n$. 
	
	Lemma \ref{lem:weakStarWeak}, again together with the continuity of $\Pi_{X^\perp}$, now implies that $H\mu_n \to H\hat{\mu}=b$.  If we put $b_n = A L^+\mu_n $, $b_n$ is in the range of the $n$-th discretized $A$-operator,  $b_n \to b$, and $\norm{\mu_n}_{TV} \geq \hat{J}_n$. 	 
	This implies
	\begin{align*}
		 \liminf_{n \to \infty} \hat J_n & = \liminf_{n \to \infty} \norm{\mu_n}_{TV} \leq \norm{\hat \mu}_{TV} = \hat J.
	\end{align*}
\end{proof}

We may now prove the main result of this section.

\begin{proof}[Proof of Theorem \ref{thm:structure}]
 By definition $b \in \ran H$. We can hence apply Lemma \ref{lem:rightDiscrete} to construct a sequence  $(b_n)_{n \in \N}$ having the properties stated in the mentioned Lemma.
 
 Now consider the problems \eqref{eq:Pn}. If we write them down explicitely, we see that the minimization over the vectors $c^{(n)}$ are exactly as in Lemma \ref{lem:finiteDimensionalStructure}, with $G=H$ and $m=\overline{m}$. Hence, we can construct a sequence $(\hat{c}_n)$ of solutions, where $\hat c_n$ containing $p_n\leq \overline{m}$ nonzero components for $n\geq \overline{m}$. 
 Thus, we may write
 \begin{align*}
 	\sum_{k\in \Omega_n} \hat c_{n,k} \delta_{k} =  \sum_{\ell=1}^{\overline{m}} d_{n,\ell} \delta_{x_{n,\ell}},
 \end{align*}
 for some $d_n \in \R^{\overline{m}}$ and $X_n=(x_{n,l})_l \in \Omega^{\overline{m}}$. In case $p_n<\overline{m}$, we may repeat positions in the vector $X_n$.

  Now $(d_n)_{n \in \N}$ is bounded, since $\norm{d_n}_1 \leq \hat J_n \leq \hat J_1$ for each $n$. This implies that there exists a subsequence, which we do not rename, such that $d_n$ is converging to $d^* \in \R^{\overline{m}}$. By possibly considering a subsequence of this subsequence, we may assume that $X_n$ converges in $\overline{\Omega}^{\times}$, where $\overline{\Omega}^{\times}$ denotes the one-point-compactification $\Omega$. This means that each of the component sequences $(x_{n,\ell})_{n}$ either converges to a point $x_\ell^*$ in $\Omega$, or diverges to $\infty$, meaning that it escapes every compact subset of $\Omega$.
 
 Consequently, the subsequence $\mu_n = \sum_{\ell=1}^{\overline{m}} d_{n,\ell} \delta_{x_{n,\ell}} \wstarto \sum_{\ell=1}^{\overline{m}} d_\ell^* \delta_{x_\ell^*}=: \mu^*$, where we identify $\delta_{\infty}$ with the zero measure (note that if $x_{n,\ell} \to \infty$, then $\delta_{x_{n,\ell}}\wstarto 0$)).

Lower semi-continuity of the $TV$-norm implies
 \begin{align*}
 	\norm{\sum_{\ell=1}^{\overline{m}} d_\ell^* \delta_{x_\ell^*}}_{TV} 
 	\leq \liminf_{n \to \infty}   \norm{\sum_{\ell=1}^p d_\ell \delta_{x_{n,\ell}}}_{TV}& = \liminf_{n \to \infty} \hat J_n  \leq \hat J,
 \end{align*}
 where we used Lemma \ref{lem:discretization} in the final step. Also, applying Lemma \ref{lem:weakStarWeak} together with the properties of $(b_n)$, we get
 \begin{align*}
 	H\mu^*  = \lim_{n \to \infty} H\mu_n = \lim_{n \to \infty} b_n  =b.
\end{align*}  
Hence, $\sum_{\ell=1}^{\overline{m}} d_\ell^* \delta_{x_\ell^*}$ is a solution of \eqref{eq:mainproblem}, which was exactly what was needed to be proven. (Note that any $x_\ell^* = \infty$ will only cause the linear combination of $\delta$-peaks  to be shorter).
\end{proof}

\subsection{Numerical Resolution} \label{sec:NumProofs}

In this section, we prove the propositions stated in Section \ref{sec:Num}. We begin with the one describing the dual problem of \eqref{eq:mainproblem}.

\begin{proof}[Proof of Proposition \ref{prop:dual}]
Define $g: \B \to \M$ with $g(u):= \|Lu\|_{TV}$. Then $J(u) = f_b(Au) + g(u)$. Standard duality arguments \cite[p.60]{ekeland1999convex} yield:
\begin{equation}
 \min_{u\in \B} J(u) = \sup_{q \in \R^m} - g^*(-A^*q) - f_b^*(q).
\end{equation}
Now, we have:
\begin{align*}
 g^*(z) & = \sup_{u \in \B} \langle z,u\rangle - g(u) \\
 & = \sup_{u \in \B} \langle z,u\rangle - \|Lu\|_{TV} \\
 &=\sup_{v\in V,  u_K \in \ker L} \langle z,v + u_K\rangle - \|Lv\|_{TV} \\
 &= \begin{cases}
     \displaystyle \sup_{v\in V} \langle z,v \rangle - \|Lv\|_{TV} & \textrm{ if } z \in (\ker L)^{\perp}, \\
     + \infty & \textrm{ otherwise}
    \end{cases}\\
&= \begin{cases}
     \displaystyle \sup_{v\in V} \langle z,L^+Lv \rangle - \|Lv\|_{TV} & \textrm{ if } z \in \ran L^*, \\
     + \infty & \textrm{ otherwise}
    \end{cases}   \\
&= \begin{cases}
     \displaystyle \sup_{w\in \ran L} \langle (L^+)^* z, w \rangle - \|w\|_{TV} & \textrm{ if } z \in \ran L^*, \\
     + \infty & \textrm{ otherwise}
    \end{cases}  \\
&= \begin{cases}
     h^*((L^+)^* z) & \textrm{ if } z \in \ran L^*, \\
     + \infty & \textrm{ otherwise}
    \end{cases} 
\end{align*}
We used the closed range theorem, which in particular implies that $\ran L^* = (\ker L)^\perp$ for an operator $L$ with closed range.

For the special case of $\ran L = \M$, we note that
\begin{align*}
	h^*(\phi) = \sup_{\mu \in \calM} \sprod{\phi, \mu} - \norm{\mu}_{TV} = \begin{cases} 0 & \textrm{ if } \norm{\phi}_{\infty} \leq 1. \\ 
	+\infty  & \textrm{ otherwise,} \end{cases}
\end{align*}

Note that the subdifferential of $g$ at every $u\in \B$ reads $\partial g(u) = L^*\partial({\|\cdot\|_{TV}})(Lu)$ (see e.g. \cite[Prop.5.7]{ekeland1999convex}.
The duality relations also follows from standard arguments, see e.g. \cite[p.60]{ekeland1999convex}.
\end{proof}


Next, we prove the proposition describing how to construct a primal solution from a dual one in the case that $\ran L= \M$.
\begin{proof}[Proof of Proposition \ref{prop:dualtoprimal}]
 We have for any operator $L$ obeying assumption \ref{assump:L}
 \begin{equation*}
 	(L^+)^*L^*= (LL^+)^* = \Pi_{\ran L}^* = \Id, \ L^*(L^+)^* = (L^+L)^* = j_V^*.
 \end{equation*}
 By construction, $A^*\hat q$ and  $L^* \partial({\|\cdot\|_{TV}})(L\hat u)$ are elements of $\ran L^*$. Due to the closed range theorem, $\ran L^*$ is isomorphic to the annihilator $V$. On that space, $j^*_V$ is injective. Hence, the inclusion \eqref{eq:primaldual} is equivalent to
 \begin{equation} \label{eq:dualreveal}
 (L^+)^*(A^*\hat q) \in \partial(\|\cdot\|_{TV})(L\hat u).
 \end{equation} 

Now, it is well known (see for instance \cite{duval2015exact}), that for all $\mu \in \M$,
\begin{equation} \label{eq:SubDiff}
 \partial(\|\cdot\|_{TV})(\mu)=\left\{\eta \in \M^*, \|\eta\|_{\infty}\leq 1, \int_{\Omega} \eta(t)\,d\mu(t) = \|\mu\|_{TV}\right\}.
\end{equation}
Consequently, \eqref{eq:dualreveal} tells us that the continuous function $(L^+)^*(A^*\hat q)$ has modulus $1$ $L\hat u$-almost everywhere on $\supp (L\hat u)$. This means that $$(L \hat u) ( I(\hat q) \backslash \supp(L \hat u)) = 0.$$ In particular, if the set $I$ only consists of isolated points, we get $\supp (L \hat u ) \subseteq I$. Hence, there exists $(d_k)_{1 \leq k \leq p}$ with
\begin{equation}
 L \hat u = \sum_{k=1}^p d_k \delta_{x_k} \, \Longrightarrow \,  \hat u = u_K + \sum_{k=1}^p d_k L^+\delta_{x_k}
\end{equation}
for some $u_K \in \ker L$.
\end{proof}

Next, we prove the claim about the structure and possible numerical resolution of the optimal measure $\hat{\mu}$ in the case of piecewise linear $\rho_i$
\begin{proof}[Proof of Proposition \ref{prop:PiecewiseLinear}]
Let $\hat{q}$ , and  $L^+\mu^*+ u_K$ be any solution of $\min_{u \in \B} J(u)$ with $\supp \ \mu^* \sse \cup_{\ell=1}^n F_\ell$, where $F_\ell$ are the faces described above (such a solution exists due to Theorem \ref{thm:structure}). Standard duality arguments (see for instance \cite[prop. 4.1]{ekeland1999convex}) yield that $\hat{q}$ and $L^+\mu^*+ u_K$ satisfies the primal-dual conditions \ref{eq:primaldual}, i.e. in particular \ref{eq:dualreveal}, since $\ran L = \M$.

 It is clear that any atomic measure $\mu= \sum_{j=1}^n d_j \delta_{\overline{x}_j}$ with $\sprod{\rho_i, \mu^*} = \sprod{\rho_i, \mu}$ for each $i$ and $\norm{\mu}_{TV} = \norm{\mu^*}_{TV}$, $L^+\mu^*+ u_K$  also is a solution to $\min J(u)$. Such a measure can be constructed as follows: Suppose that there exists a face $P$ of at least dimension $1$ of a polytope $F_j$ such that $\supp \ \mu_*$ intersects in at least one point $p$ of the relative interior of $P$ (if no such $P$ exists, $\mu_*$ is already atomic). Due to $\supp \mu_* \sse I(\hat{q})$, $(L^+)^*A^*q$ has absolute value $1$ in $p$. $(L^+)^*A^*q$ being a continuous piecewise linear function with absolute value bounded by one, it must therefore have a constant value $\varepsilon$, either equal to $+1$ or $-1$, on $P$. Due to the structure \eqref{eq:SubDiff} of the subdifferential of the $TV$-norm, this implies that $\mu_*$ (the unimodular part of the polar decomposition of $\mu_*$ to be exact) must have the same sign as $\epsilon$ almost everywhere on $P$.
 
 On $P$, each function $\rho_i$ can be written as $\rho_i(x) = \sprod{\alpha_{ij}, x} +\beta_{ij}$, for some vectors $\alpha_{ij} \in \R^d$ and scalars $\beta_{ij} \in \R$. If we hence define
$$ d = \epsilon\abs{\mu}^*(P), \quad \overline{x}= \frac{1}{\abs{\mu}^*(P)}\int_{P} x d\abs{\mu}^* \in P, $$
and $\mu = d\overline{x} + \mu^* \vert_{\Omega \backslash P}$, we have $\norm{\mu}_{TV} =  \abs{d}  + \norm{\mu^*\vert_{\Omega \backslash P}} =  \norm{\mu^*\vert_P}_{TV} + \norm{\mu\vert_{\Omega \backslash P}}= \norm{\mu^*}_{TV}$, and
\begin{align*}
	\sprod{\rho_i, \mu} = \sum_{j=1}^n \left( \sprod{\alpha_{ij},\overline{x}_j} + \beta_{ij} \right) \mu^*(F_j) = \sum_{j=1}^n \int_{F_j} \left( \sprod{\alpha_{ij},x} + \beta_{ij} \right) d\mu^* = \sprod{\rho_i, \mu^*}.
\end{align*}
By iteratively removing all such non-atomic parts of $\mu_*$, we obtain an atomic solution $\mu$.

We still need to prove that we can find a $\mu^*= \widetilde{\mu}$ which is atomic and supported on the vertices of $F_j$. Note that each $\overline{x}_j$ can be represented as a convex combination $\sum_{k=1}^{t_j} \theta_k v_{jk}$ of the vertices $v_{jk}$ of $F_j$. Defining a measure
\begin{align*}
	\widetilde{\mu} = \sum_{j=1}^n \sum_{k=1}^{t_j} \theta_k \mu(F_j) \delta_{v_{jk}},
\end{align*}
we see that $\norm{\mu}_{TV} = \Vert \widetilde{\mu} \Vert_{TV}$ and 
\begin{align*}
\langle \rho_i, \widetilde{\mu}\rangle = \sum_{j=1}^n  \sum_{k=1}^{t_j} \theta_k \left( \sprod{\alpha_{ij},v_{jk}} + \beta_{ij} \right) \mu(F_j) = \sum_{j=1}^n \left( \sprod{\alpha_{ij},\overline{x}_j} + \beta_{ij} \right) \mu(F_j) = \sprod{\rho_i, \mu},
\end{align*}
so that $\widetilde{\mu}$ is also a solution.

\end{proof}

Finally, we provide the argument that the constraint of the dual problem can be rewritten as an inequality on the space of Hermitian matrices in the case of the functions $\rho_i$ begin trigonometric polynomials.
\begin{proof}[Proof of Lemma \ref{lem:trigo}]
 Note that $|\sum_{i=1}^m \alpha_i \rho_i|\leq 1$ is equivalent to
 \begin{align*}
  1& \geq \abs{\sum_{i=1}^m \alpha_i \sum_{j=-K}^K \gamma_{i,j} p_j} = \abs{\sum_{j=-K}^K \sum_{i=1}^m \alpha_i  \gamma_{i,j} p_j} =\abs{\sum_{j=-K}^K (\Gamma\alpha)_j p_j} =\abs{p_{-K}\sum_{j=-K}^K (\Gamma\alpha)_j p_j}.
 \end{align*}
The function $f=p_{-K}\sum_{j=-K}^K (\Gamma\alpha)_j p_j$ is a causal trigonometric polynomial. 
We know from \cite[Cor.4.27]{dumitrescu2007positive} that it obeys the constraint $\|f\|_\infty\leq 1$ if and only if there exists a positive semi-definite matrix $Q\in \C^{(2K+1)\times (2K+1)}$ such that
\begin{equation}
\begin{bmatrix}
Q & \Gamma \alpha \\
(\Gamma \alpha)^* & 1
\end{bmatrix}\succeq 0 \textrm{ and }
\sum_{i=1}^{2K+2-j} Q_{i,i+j}=
\begin{cases}
1, & j=1, \\
0, & 2\leq j \leq 2K+1.                                                                                                
\end{cases}
\end{equation}
\end{proof}

\subsection{Differential operators of Section \ref{ex:DiffOp}}

Here, we provide the proofs for the lemmas and propositions which include more general differential operators into our framework. We begin by proving a preparatory lemma about the operator $L^+$ in \eqref{eq:LPlusDiffOp}.
\begin{lemma} \label{lem:LPlusBP}
	The operator $L^+$ defined by \eqref{eq:LPlusDiffOp} is a continuous operator from $\M(\Omega)$ to $\calC(\Omega)$. It has the property $P(D) L^+ =\Id_{\M(\Omega)}$.
\end{lemma}
\begin{proof}
Let us begin by showing that $L^+$ maps from $\M(\Omega)$ to $\calC(\Omega)$. First, note that the continuity of $x \mapsto u_x$ implies that $L^+\mu$ is pointwise well-defined. We still need to show that for a fixed $\mu$, the map $x \mapsto (L^+ \mu)(x)$ is continuous. This follows from a standard ``limits and integrals commute'' argument. Let $x_n \to x$. Then $u_y (x_n) \to u_y(x)$ pointwise. Furthermore, $\abs{u_y(x_n)} \leq \sup_{y \in \Omega} \norm{u_y}_\infty$ for all $y$ and $x_n$. Since $\sup_{y \in \Omega} \norm{u_y}_\infty$ is a $\mu$-integrable function, the theorem of Lebesgue implies that
\begin{align*}
	\lim_{n\to \infty}  (L^+\mu)(x_n)\lim_{n\to \infty} \int_{\Omega} u_y(x_n) d\mu(y) = \int_\Omega u_y(x) d\mu(y) = (L^+\mu)(x).
\end{align*} 
The boundedness of the map now follows from the inequality
\begin{align*}
	\abs{ \int_\Omega u_y(x) d\mu} \leq  \int_\Omega \abs{u_y(x)} d\abs{\mu} \leq \sup_{y \in \Omega} \norm{u_y} \norm{\mu}_{TV}, \ x \in \Omega.
\end{align*}

Now we show that $P(D)L^+\mu = \mu$. For this, let $\phi \in \calC_c^\infty(\Omega)$ be arbitrary. We then have
\begin{align*}
	\int_{\Omega}(L^+\mu)(y)  P(D_y)^*\phi(y) dy = \int_\Omega \int_\Omega u_x(y) P(D_y)^*\phi(y) d\mu(x) dy,
\end{align*}
where $P(D)^*$ denotes the adjoint to $P(D)$. The function $(x,y) \mapsto u_x(y) P(D_y)^* \phi(y)$ is continuous and supported on a set of the form $\Omega \times C$, where $C$ is compact. As such, it is integrable with respect to the measure $\mu \otimes dy$, and we may apply Fubini's theorem. Subsequently shifting $P(D_y)$ onto $u_x$ and utilizing $P(D_y) u_x = \delta_x$, we obtain that the above is equal to
\begin{align*}
	\int_\Omega \phi(x) d\mu(x).
\end{align*}
This exactly means that $P(D) L^+\mu = \mu$.
\end{proof}

Now we may prove Lemma \ref{lem:BP} about the properties of $\B_P$ as a normed space.

\begin{proof}[Proof of Lemma \ref{lem:BP}]
The only non-trivial step in proving that $\norm{u}_{\B_P}$ is a norm is to prove that $\norm{u}_P = 0 \Rightarrow u=0$. This follows from the assumption on the set $K$: If $\norm{u}_{B_P}=0$, then in particular $P(D)u=0$ and $u=0$ in $K$. Since $u=0$ is a function obeying $P(D)u=0$ and $u=0$ in $K$, the uniqueness of the continuation implies that $u$ must vanish everywhere in $\Omega$.

To prove that $\B_P$ is a Banach space, notice that we can interpret $\B_P$ as a subspace of the Banach space $\calM(\Omega) \times \calM(K)$. This space is furthermore closed: If $(P(D)u_n, u_n) \to (\mu, \overline{u})$ in $\calM(\Omega) \times \calM(K)$, there must be $P(D)\overline{u} = \mu$ on $K$. To see this, let $\phi \in \calC_0(K)^\infty$ be arbitrary. We then have $P(D)^*\phi \in \calC_0(K)$, and consequently
\begin{align*}
	\int_{K} \overline{u} P(D)^*\phi dx =  \lim_{n\to \infty} \int_K u_n P(D)^*\phi dx = \lim_{n\to \infty} \int_K  \phi d (P(D)u_n)  = \int_K \phi d \mu,
\end{align*}
where we used the fact $P(D) u_n \to \mu$ in the last step. Since $P(D)L^+\mu = \mu$, we conclude that $P(D)(\overline{u} - L^+\mu) =0$ in $K$. The continuation property implies that there exists a $\widehat{u}$ with $P(D)\widehat{u} = 0$ in $\Omega$ and $\widehat{u} = \overline{u} - L^+\mu$ in $K$. We then have $\mu = P(D)(\widehat{u} + L^+\mu)$ and $(\widehat{u} + L^+\mu )\vert_K = \overline{u}$, so that $(\mu, \overline{u}) \in \B_P$.
\end{proof}

Now let us prove Proposition \ref{prop:DiffOp}

\begin{proof}[ Proof of Proposition \ref{prop:DiffOp}]
	In Lemma \ref{lem:LPlusBP}, we showed that  $P(D)L^+ = \Id_{\M}$. This already proves that  $\ran P(D)=\M(\Omega)$. Also, it shows that $L^+$ is a continuous operator from $\M(\Omega)$ to $\B_P(\Omega)$: $\calC(\Omega) \hookrightarrow \M(K)$ due to the boundedness of $K$, and if $\mu_n \to \mu$ in $\B_P$, then $P(D)L^+\mu_n = \mu_n \to \mu = P(D) L^+\mu$.
 
 It follows that $L^+P(D) =\Id_{\ran{L^+}}$. If we can prove that $\ran L^+$ is closed, we have shown that $\ker P(D)$ has the closed complementary subspace $\ran L^+$.

To show the latter, let $u_n = L^+\mu_n$ in $\ran L^+$ converge to an element $\overline{u} \in \B_P$. Then, by definition of $\B_P$, $P(D)L^+\mu_n = \mu_n \to P(D)\overline{u}$. Consequently, by the continuity of $L^+$,
\begin{align*}
	u_n = L^+ \mu_n = L^+ P(D) L^+\mu_n \to L^+P(D) \overline{u},
\end{align*}
so that $\overline{u} = L^+ P(D) \overline{u} \in \ran L^+$.

It remains to calculate the operator $(L^+)^*$. For $a \in L^1(\Omega)$ and $\mu \in \M(\Omega)$, we have
\begin{align*}
	\sprod{(L^+)^*a,\mu} = \sprod{a, L^+\mu} = \int_\Omega \int_\Omega a(y) u_x(y) d\mu(x) dy
\end{align*}
$(x,y) \to u_x(y) a(y)$ is in $L^1(\mu \otimes dy)$, so that we may apply Fubini and obtain
\begin{align*}
	 \sprod{(L^+)^*a,\mu}\int_\Omega \left(\int_\Omega a(y) u_x(y) dy \right) d\mu(x) .
\end{align*}
The last assertion about $(L^+)^*a \in \calC_0(\Omega)$ is argued as follows. Let $\epsilon>0$. First, since $a \in L^1(\Omega)$, there exists a compact set $C$ such that $\norm{a_{\Omega \backslash C}}_1 \leq \epsilon$. Further, since   the map $x \mapsto u_x$ is vanishing at infinity as a map from $\Omega$ to $\calC(\Omega \cap B_R(0))$, there exists a compact set $\widetilde{C}$ such that if $x \notin \widetilde{C}$, $\norm{u_x}_{\calC(C)}\leq \epsilon$. This implies for such $x$
\begin{align*}
	\abs{(L^+)^*a(x)} &\leq  \norm{u_x}_\infty \int_{\Omega \backslash C}  \abs{a(y)} dy +  \norm{u_{x}}_{\calC(C)} \int_{C}  \abs{a(y)} dy \\
	&\leq \norm{u_x}_\infty \epsilon + \epsilon \norm{a}_1,
\end{align*}
so that the theorem is proved.
\end{proof}

Now let us finally argue that the differential operators of the form \eqref{eq:Elliptic} can be included in our framework.

\begin{proof}[Proof of Proposition \ref{prop:Elliptic}]
Consider the  space $\calH^k_0(\Omega)$, defined as the closure of $\calC_0^\infty(\Omega)$ in the Sobolev norm $\norm{\cdot}_{H^k(\Omega)}$. We can formulate the problem $P(D)u= f$ as an operator equation on $\calH_0^k(\Omega)$ as follows:
	\begin{align*}
		\sprod{P(D)u,v} = \int_\Omega \sum_{\abs{\alpha}=k} \sum_{\abs{\beta}=k} p_{\alpha,\beta}(x) D^\alpha u(x) D^\beta(x) dx = \sprod{f, v}, v \in \calH^k_0(\Omega).
	\end{align*}
	By the Lax-Milgram lemma together with the ellipticity condition, this problem has a unique solution as soon as $f \in \calH^k_0(\Omega)$. Now, since $k> d/2$, we have the continuous Sobolev embedding $\calH^k_0(\Omega) \hookrightarrow \calC_0(\Omega)$. (For $\Omega = \R^d$, this can be proven with Fourier methods, for a bounded domain, this is a Sobolev embedding theorem.)  This both proves that $\delta_x \in \calH^k_0(\Omega)^*$ and that the solution $u_x \in \calC_0(\Omega)$.
	
	To show that the map $x \to u_x$ is vanishing at infinity on compact sets, let us first assume that $\Omega$ is bounded. When $x$ escapes to infinity, $\delta_x \wstarto 0$ in $\calM(\Omega)$, and therefore also in $\calH^k_0(\Omega)$. The ``continuous dependence on the data''-part of Lax-Milgram theorem therefore implies that $u_x \wto \, 0$ in $\calH^k_0(\Omega)$.  Since the embedding $\calH^k_0(\Omega) \hookrightarrow \calC_0(\Omega)$ in this case even is compact (see e.g. \cite[Theorem 6.2]{Adams}), this implies that $u_x \to 0$ in $\calC_0(\Omega)$, which was to be proven.
	
	Now let $\Omega=\R^d$ and $R>0$ be arbitrary. The result \cite[Theorem 10.2.1]{Hormander} states that the solution $u_x$ is equal to $\Phi (\cdot -x)$ for a $\Phi$ obeying 
	\begin{align*}
		\sup_{\xi \in \R^d} \abs{\widetilde{P}(\xi)\widehat{\Phi}(\xi)} < \infty,
	\end{align*}
	where $\widetilde{P}$ is defined as
	\begin{align*}
		\widetilde{P} = \left( \sum_{\alpha \geq 0} \abs{D^\alpha P (\xi)}^2\right)^{1/2}.
	\end{align*}
	By using the ellipticity assumption, one sees that this implies that $(1+ \abs{\xi}^{2k}) \vert \hat{\Psi} \vert \leq C$, which ensures that $\abs{\hat{\Psi}}$ is integrable ($k>d/2$). By the Riemann-Lebesgue theorem, $\Psi \in \calC_0(\R^d)$. This already implies that $u_x = \Psi(\cdot -x)$ vanishes to infinity on compact sets.
	
	To prove the final claim, let $u$ obey $P(D)u=0$ on $K$. Then in particular $u \in \calH^k(K)$. This implies that for every set $K \sse \widetilde{\Omega} \sse \Omega$ with $\mathrm{dist}( K, \partial \Omega)>0$, there exist a function $\tilde{u} \in \calH^k(\R^d)$ with compact support in $\widetilde{\Omega}$ (see \cite[Theorem 2.8]{Adams}) such that $ u = \widetilde{u}$ on $K$. Now consider the following problem:
	\begin{align*}
		\begin{cases} P(D) u = -P(D)\widetilde{u}, \ x\in \Omega \backslash K \\
D^\alpha u \vert_{\partial K \cup \partial\Omega} = 0\end{cases}.		
	\end{align*}
	This problem can be shown to have a solution $\widehat{u}$. Now consider the function
	\begin{align*}
		\overline{u}(x) = \begin{cases} u(x), x \in K \\
		-\widehat{u}(x)+ \widetilde{u}(x) , x\in \Omega \notin K \end{cases}
	\end{align*}
	Due to boundary term cancellation, together with the fact that the $\overline{u}_{K}$ and $\overline{u}_{\Omega \notin K}$ solves the problem $P(D)u$ in their respective domains, this function obeys $P(D)\overline{u}=0$ in $\Omega$, and of course $\overline{u} \vert_K = u$.


As for the uniqueness of the extension, we note that if $P(D)u=0$ on $\Omega$, the ellipticity assumption implies that $\sum_{\alpha=k} \abs{D^\alpha u}_2^2 =0$. This in particular implies that $\Delta^k u =0$, i.e., $\Delta^{k-1} u$ is harmonic. Since $\Delta^{k-1}u$ vanishes on $K$, and $K$ has non-empty interior, it must vanish everywhere (this is the identity theorem of harmonic functions). By repeating this argument $k$ times, we finally obtain that $u$ vanishes on $\Omega$.
\end{proof}

\subsection{Miscellaneous}\label{sec:misc}

Here, the rest of the left out proofs are given. We start with the simple proposition about existence of solutions to the problem \eqref{eq:mainproblem}.

\begin{proof}[Proof of Proposition  \ref{prop:solvability}.]
Let $(u_n)_{n \in \N}$ be a minimizing sequence for \eqref{eq:mainproblem}. Let us write $u_n = L^+\mu_n + u_{n,K}$ with $\mu_n \in \M(\Omega)$ and $u_{n,K}\in \ker L$ for each $n \in \N$. We may thereby without loss of generality assume that $u_{n,K} \in \vect( w_\ell)_{\ell=1}^{\widehat{m}}$, where $w_\ell$ are vectors such that $(Aw_\ell)_{\ell=1}^{\widehat{m}} $ spans $A(\ker L)$ (any alteration of $u_{K,n}$ not parallel to this space will neither change the value of $\norm{Lu}$ or the value of $f_b(Au)$. 

Now, due to the minimization property of the sequence, $$(\mu_n)_{n \in \N}\text{ and }(f_b(A(L^+\mu_* + u_{n,K}))_{n \in \N}$$ are both bounded. Due to the coercivity of $f_b$ together with the fact that $A$ restricted to the space $(w_\ell)_{\ell=1}^{\widehat{m}}$ is injective, the sequence $(u_{n,K})_{n \in \N}$ will be bounded in $A(\ker L)$. Due to the Banach-Alaoglu theorem and the separability of $\calC_0(\Omega)$ (i.e. the pre-dual of $\M(\Omega)$), $(\mu_n)_{TV}$ will contain a subsequence which converges to, say, $\mu^*$. Similarly, since $(u_{n,K})$ lives in the finite-dimensional space $\vect(w_\ell)_{\ell=1}^{\hat{m}}$, it will also contain a subsequence convergent to, say $u^*$. Now, using the same notation for the convergent subsequences as for the sequences themselves, we have
\begin{align*}
	\norm{\mu^*}_{TV} + f_b(A(L^+\mu^*+u_K^*)) \leq \liminf \norm{\mu_n}_{TV} + f_b( A(L^+\mu_n + u_{n,K})) \\
	=\liminf \norm{Lu_n}_{TV} + f_b( Au_n) =\min_{u \in \B} \norm{Lu}_{TV} + f_b(Au).
\end{align*}
We used Lemma \ref{lem:weakStarWeak} and the lower semicontinuity of $f_b$ and of the $TV$-norm. Hence, $L^+\mu^* + u^*_K$ is the solution whose existence we had to prove.
\end{proof} 

Now let us include spline-admissible operators in our framework.

\begin{proof}[Proof of Lemma \ref{prop:Unser}]
1. The finite-dimensionality of $\ker L$ is simply assumption 3 of Theorem 1 of \cite{unser2016splines}. Theorem 4 and 5 of \cite{unser2016splines} proves that $L$ has a right inverse $L_\Phi^{-1}$. This implies that
\begin{align*}
	\ran L \sse \ran LL_\Phi^{-1} = \ran \Id = \M.
\end{align*}

2. The space $\calC_L$ as defined in Theorem 6 of \cite{unser2016splines} is defined as
\begin{align*}
	\calC_L = L^* (\calC_0(\R^d)) + \vect(\phi_i)_{i =1}^r,
\end{align*}
where $\phi_i$ is a system of functionals which restricted to $\ker L$ becomes a of the dual of $\ker L$. Without loss of generality, we can assume that $\phi_i \vert_V =0$ for each $i$ (if not, we could instead consider the operators $\widetilde{\phi}_i = \phi_i \Pi_{\ker L}$).

Then if $a \in \calC_L$, we have
\begin{align*}
	(L^+)^*a = (L^+)^*L^*\rho + \sum_{i=1}^r \gamma_i (L^+)^* \phi_i
\end{align*}
for some $\rho \in \calC_0(\R^d)$ and $\gamma_i$. Now  $(L^+)^*L^*) = (LL^+)^* = \Pi_{\ran L}^* = \Id$ and $(L^+)^* \phi_i =0$, so that $(L^+)^*a = \rho \in C_0(\R^d)$.

If on the other  $(L^+)^*a \in \calC_0(\R^d)$, we have 
\begin{align*}
	L^*\calC_0(\R^d) \ni L^*(L^+)^*a = (L^+L)^* a = \Pi_V^*a. 
\end{align*}
Since each functional $a \in \M_L^*$ can be written as $\Pi_V^* a + \Pi_{ \ker L}^*a$, and $\Pi_{\ker L}^* a \in \vect(\phi_i)_{i=1}^r$, $a \in \C_L$.
\end{proof}

Next, we discuss the case of $L$ being the differential operator on $BV((0,1)$.
\begin{proof}[Proof of Lemma \ref{lem:Dplus}]
Note that we have $\ker L=\vect(1)$, the vector space of constant functions on $\Omega$, hence the space $V$ can be identified with the space of functions with zero mean: 
\begin{equation*}
V=\left\{u \in BV(\Omega), \int_{\Omega} u(t)\,dt=0\right\}. 
\end{equation*}

For $\mu \in \M$, consider the mapping $I:\mu \mapsto u$ defined for $s\in [0,1]$ by $u(s)=\mu([0,s])$.
We only need to prove that $DI(\mu) = \mu$ in the distributional sense. Let $\phi \in C^\infty_c(\Omega)$:
\begin{align*}
\langle I(\mu) , \phi'\rangle & = \int_{0}^1 \mu([0,t]) \phi'(t)\, dt \\
&= \int_{0}^1 \int_{0}^1 1_{[0,t]}(s) d\mu(s) \phi'(t)\, dt \\
&= \int_0^1 \int_{0}^1 1_{[s,1]}(t)  \phi'(t)\, dt d\mu(s) \\
&=\int_0^1 -\phi(s) d\mu(s)= -\langle \mu , \phi\rangle.
\end{align*}
This proves the surjectivity of $L$. We see that the proposed form of $L^+$ is the right one, since $s \mapsto \mu([0,s]) - \int_0^1 \mu([0,s])ds$ is a function of zero mean.

We now calculate
\begin{align*}
\sprod{(L^+)^*\xi, \mu}& = \sprod{\xi, L^+ \mu} \\
&= \int_0^1 \xi(t) \left(\int_0^1 1_{[0,t]}(s) d\mu(s) - \int_0^1 \mu([0,r])  dr \right) dt \\
&= \int_0^1 \left(\int_0^1  1_{[s,1]}(t)\xi(t) dt\right)d\mu(s) - \int_0^1 \xi(t) dt\cdot \int_0^1 1_{[0,r]}(s) d\mu(s)  dr \\
&= \int_0^1 \left(\int_s^1  \xi(t) dt \right) d\mu(s) -\int_0^1 \xi(t) dt\cdot \int_0^1 (1-s)  d\mu(s) 
\end{align*}
In particular, the action of $(L^+)^*\xi$ is given by a continuous function, which is vanishing on the boundary of $(0,1)$

\end{proof}

\section{Conclusion \& Outlook}
In this paper we have studied the properties of total variation regularized problems, where total-variation should be understood as a term of form  $\|Lu\|_{TV}$, with $L$ a linear operator.
We have shown that under a convexity assumption on the data-fit term, some of the solutions $\hat u$ of total-variation regularized inverse problems are $m$-sparse, where $m$ denotes the number of measurements. 
This precisely means that $L\hat u$ is an atomic measure supported on at most $m$ points.
This result extends recent advances \cite{unser2016splines}, by relaxing some hypotheses on the linear operator $L$ and on the domain of the functions.

The second contribution of this paper is to show that solutions of this infinite dimensional problem can be obtained by solving one or two consecutive finite dimensional problems, given that the measurements belong to some function spaces such as the trigonometric polynomials or the set of piecewise linear functions on polyhedral domains. Once again, this result extends significantly recent results on super-resolution \cite{candes2014towards,tang2013compressed}.
The analysis provided for piecewise linear functions is novel and we believe that it might have important consequences in the numerical analysis of infinite dimensional inverse problems: the scaling with respect to the number of grid points is just linear, contrarily to approaches based on semi-definite relaxations or Lasserre hierarchies.  

As an outlook, we want to stress out that the hypotheses formulated on the linear operator $L$ rule out a number of interesting applications, such as total variation regularization in image processing. We plan to study how the results and the proof techniques in this paper could apply to more general cases. 

\section*{Acknowledgement}
A. Flinth acknowledges support  from the Deutsche Forschungsgemeinschaft (DFG) Grant KU 1446/18-1, and from the Berlin Mathematical School (BMS). He also wishes to thank Yann Traonmillin, Felix Voigtländer and Philipp Petersen for interesting discussions on this subject.
This work was partially funded by ANR JCJC OMS. 
P. Weiss wishes to thank Michael Unser warmly for motivating him to work on the subject at the second OSA ``Mathematics in Imaging'' conference in San Francisco and for providing some insights on his recent paper \cite{unser2016splines}.
In addition, he thanks Didier Henrion particularly and Alban Gossard, Frédéric de Gournay, Jonas Kahn, Etienne de Klerk, Jean-Bernard Lasserre, Michael Overton and Lieven Vandenberghe for interesting feedbacks and insights on a preliminary version of this work.
The two authors wish to thank Gitta Kutyniok from TU Berlin for supporting this research.



\bibliographystyle{abbrv}
\bibliography{Biblio}

\end{document}